\documentclass[11pt]{article}
\usepackage{amsmath,amssymb,textcomp,stmaryrd,xifthen,psfrag,graphicx,color}
\usepackage{amsfonts,amsthm}
\usepackage{color}
\usepackage{enumerate}

\SetSymbolFont{stmry}{bold}{U}{stmry}{m}{n} 
\usepackage[T1]{fontenc}
\date{}

\newcommand{\eremk}{\hbox{}\hfill\rule{0.8ex}{0.8ex}}

\newtheorem{theorem}{Theorem}[section]
\newtheorem{lemma}[theorem]{Lemma}
\newtheorem{cor}[theorem]{Corollary}
\newtheorem{proposition}[theorem]{Proposition}
\theoremstyle{definition} \newtheorem{definition}[theorem]{Definition}
\newtheorem{remark}[theorem]{Remark}

\newcommand{\vdual}[2]{(#1\hspace*{.5mm},#2)}
\newcommand{\abs}[1]{\vert #1 \vert}
\newcommand{\norm}[3][]{#1\|#2#1\|_{#3}}
\newcommand{\snorm}[2]{|#1|_{#2}}
\newcommand{\enorm}[2][]{#1|\hspace*{-.3mm}#1|\hspace*{-.3mm}#1|#2#1|\hspace*{-.3mm}#1|\hspace*{-.3mm}#1|}

\newcommand{\diam}{\mathrm{diam}}
\newcommand{\wilde}{\widetilde}
\newcommand{\wat}{\widehat}
\newcommand{\dist}{{\rm dist}}

\def\div{{\rm div}}

\newcommand{\trace}{ \operatorname*{tr}}

\newcommand{\supp}{{\rm supp}}

\DeclareMathOperator*{\argmin}{arg\,min}

\newcommand{\cA}{\ensuremath{\mathcal{A}}}
\newcommand{\cL}{\ensuremath{\mathcal{L}}}

\newcommand{\R}{\ensuremath{\mathbb{R}}}
\newcommand{\N}{\ensuremath{\mathbb{N}}}

\newcommand{\dD}{\ensuremath{\mathcal{D}}}
\newcommand{\hH}{\ensuremath{\mathcal{H}}}
\newcommand{\eE}{\ensuremath{\mathcal{E}}}

\newcommand{\Ii}{\ensuremath{\mathcal{I}}}
\newcommand{\lL}{\ensuremath{\mathcal{L}}}
\newcommand{\kK}{\ensuremath{\mathcal{K}}}

\newcommand{\bb}{\ensuremath{\mathbf{b}}}

\newcommand{\ww}{\ensuremath{\boldsymbol{w}}}
\newcommand{\tT}{\ensuremath{\mathcal{T}}}

\newcommand{\sS}{\ensuremath{\mathcal{S}}}

\newcommand{\fu}{\ensuremath{\mathbf{u}}}

\newcommand{\fw}{\ensuremath{\mathbf{w}}}

\newcommand{\vV}{\ensuremath{\mathbf{v}}}

\newcommand{\xx}{\ensuremath{\mathbf{x}}}

\newcommand{\zz}{\ensuremath{\mathbf{z}}}

\newcommand{\Af}{\ensuremath{\mathbf{A}}}
\newcommand{\Bf}{\ensuremath{\mathbf{B}}}
\newcommand{\Xf}{\ensuremath{\mathbf{X}}}
\newcommand{\Yf}{\ensuremath{\mathbf{Y}}}
\newcommand{\If}{\ensuremath{\mathbf{I}}}

\newcommand{\BL}{\ensuremath{\mathcal{B}}}
\newcommand{\calH}{\ensuremath{\mathcal{H}}}

\newif\iftechreport
\techreportfalse
\iftechreport
\else
\fi

\iftechreport
\title{${\mathcal H}$-matrix approximability of inverses of discretizations of the fractional Laplacian
(extended version)
\thanks{MK was supported by Conicyt Chile through project FONDECYT 1170672.
JMM was supported by the Austrian Science Fund (FWF) project F 65.
}}
\else
\title{${\mathcal H}$-matrix approximability of inverses of discretizations of the fractional Laplacian
\thanks{MK was supported by Conicyt Chile through project FONDECYT 1170672,
JMM was supported by the Austrian Science Fund (FWF) project F 65.
}}
\fi

\author{Michael Karkulik\thanks{Departamento de Matem\'atica, Universidad T\'ecnica Federico Santa Mar\'ia,
  Avenida Espa\~na 1680, Valpara\'iso, Chile,
  \texttt{mkarkulik.mat.utfsm.cl}, email: \texttt{michael.karkulik@usm.cl}}
  \and Jens Markus Melenk\thanks{Institut f\"ur Analysis und Scientific Computing, Technische Universit\"at Wien,
    Wiedner Hauptstrasse 8-10, Wien, Austria, 
email: \texttt{melenk@tuwien.ac.at}}}
\begin{document}
\maketitle
\begin{abstract}
The integral version of the fractional Laplacian on a bounded domain is discretized by a 
Galerkin approximation based on piecewise
linear functions on a quasi-uniform mesh. We show that the inverse of the associated stiffness 
matrix can be approximated by blockwise low-rank matrices at an exponential rate
in the block rank. 

\bigskip
\noindent
{\em Key words}: Hierarchical Matrices, Fractional Laplacian\\
\noindent
{\em AMS Subject Classification}: 65N30, 65F05, 65F30, 65F50
\end{abstract}
\section{Introduction}
Fractional differential operators are non-local operators with 
many applications in science and technology and interesting mathematical
properties; a discussion of some of their features can be found, e.g., in  
\cite{di-nezza-palatucci-valdinoci12}. The nonlocal nature of such operators
implies for numerical discretizations that the resulting 
system matrices are fully populated. Efficient matrix compression techniques 
are therefore necessary. Various data-sparse representations of 
(discretizations) of classical integral operators have been proposed in the 
past. We mention techniques based on 
multipole expansions, 
panel clustering, 
wavelet compression techniques,
the mosaic-skeleton method, 
the adaptive cross approximation (ACA) method, 
and the hybrid cross approximation (HCA); 
we refer to \cite{faustmann-melenk-praetorius15b} for a more detailed
literature review in the  context of classical boundary element methods (BEM). 
In fact, many of these data-sparse methods 
may be understood as specific incarnations of
${\calH}$-matrices, which were introduced
in \cite{Hackbusch99,GrasedyckHackbusch,GrasedyckDissertation,hackbusch15}
as blockwise low-rank matrices. Although many of the above mentioned techniques
were originally developed for applications in BEM
the underlying reason for their success is the 
so-called ``asymptotic smoothness'' of the kernel function, which is given
for a much broader class of problems. We refer 
to \cite{doelz-harbrecht-schwab17} and references therein, where the question
of approximability is discussed for pseudodifferential operators. 
Discretizations of integral versions of the fractional Laplacian such as 
the one considered in the present paper, \eqref{eq:integral-fractional-laplacian}, 
are therefore amenable 
to data-sparse representations with $O(N \log^\beta N)$ complexity, where
$N$ is the matrix size and $\beta \ge 0$. This compressibility has 
recently been observed in \cite{zhao-hu-cai-karniadakis17} and in
\cite{ainsworth-glusa17}, where an analysis and implementation of a 
panel clustering type matrix-vector multiplication for the stiffness matrix is presented. 
It is the purpose of the present paper
to show that also the inverse of the stiffness matrix of a discretization 
of the integral version of the fractional Laplacian can be represented
in the $\calH$-matrix format, using the same underlying block structure 
as employed to compress the stiffness matrix. 

One reason for studying the compressibility of the inverses (or the closely
related question of compressibility of the $LU$-factors) 
are recent developments in fast (approximate) arithmetic for data-sparse 
matrix formats. For example, $\calH$-matrices come with 
an (approximate) arithmetic with log-linear complexity, which includes, 
in particular, the (approximate) inversion and factorization of matrices. 
These (approximate)
inverses/factors could either be used as direct solvers or 
as preconditioners, as advocated, for example, in a BEM context in  
\cite{bebendorf05,Grasedyck05,Grasedyck08,leborne-grasedyck06,grasedyck-kriemann-leborne08} and in \cite{li-mao-song-wang-karniadakis18} in the context
of fractional differential equations. 
We point out that the class of $\calH$-matrices is not the only one for which
inversion and factorizations algorithms have been devised. 
Related to ${\calH}$-matrices and its arithmetic are
``hierarchically semiseparable matrices'',
\cite{xia13,xia-chandrasekaran-gu-li09,li-gu-wu-xia12} and the idea of
``recursive skeletonization'',
\cite{ho-greengard12,greengard-gueyffier-martinsson-rokhlin09,ho-ying13};
for discretizations of PDEs, we mention \cite{ho-ying13,gillman-martinsson13,schmitz-ying12,martinsson09}, and
particular applications to boundary integral
equations are \cite{martinsson-rokhlin05,corona-martinsson-zorin13,ho-ying13a}.


The underlying structure of our proof is similar to that in 
\cite{faustmann-melenk-praetorius15b,faustmann-melenk-praetorius15c}
for the classical single layer and hypersingular operators of BEM. There, 
it is exploited that these operators are traces of potentials, i.e., 
they are related to functions that solve an elliptic PDE. The connection 
of \cite{faustmann-melenk-praetorius15b,faustmann-melenk-praetorius15c}
with the present article is given by the works 
of \cite{CaffarelliS_07, stinga-torrea10,caffarelli-stinga16}, which show 
that fractional powers of certain elliptic operators posed in $\R^d$ can be realized 
as the Dirichlet-to-Neumann maps for (degenerate) PDEs posed in $\R^{d+1}$. 

\subsection{The fractional Laplacian and the Caffarelli-Silvestre extension}
In this section, we briefly introduce the fractional Laplacian; the discussion will
remain somewhat formal as the pertinent function spaces (e.g., $H^s_0(\R^d; \Omega)$) 
and lifting operators (e.g., $\cL$) will be defined in subsequent sections. 

For $s \in (0,1)$, the fractional Laplacian in full space
$\R^d$ is classically defined through the Fourier transform, 
$(-\Delta u)^{s}:= {\mathcal F}^{-1} \left(|\xi|^{2s} {\mathcal F}(u)\right)$. 
As discussed in the survey \cite{kwasnicki17}, several equivalent definitions 
are available. For example, for suitable $u$, a pointwise
characterization is given in terms of a principal value integral: 
$$
(-\Delta u)^s(x) = C(d,s)\, \text{P.V.} \int_{\R^d} \frac{u(x) - u(y)}{|x - y|^{d+2s}}\,dy, 
\qquad 
C(d,s):= 2^{2s}s \frac{\Gamma(s+d/2)}{\pi^{d/2}\Gamma(1-s)}. 
$$
Caffarelli and Silvestre \cite{CaffarelliS_07} characterized this operator 
as the Dirichlet-to-Neumann operator of a (degenerate) elliptic PDE. That is,
they proved 
\begin{equation}
\label{eq:CS} 
C(d,s) (-\Delta u)^s (x) = - \lim_{x_{d+1} \rightarrow 0+} 
x_{d+1}^{1-2s} \partial_{x_{d+1}} (\cL u)(x,x_{d+1}), 
\qquad x \in \R^d, 
\end{equation}
where the \emph{extension} $\cL u$ is a function on the half-space 
$\R^{d+1}_+:= \{(x,x_{d+1})\,|\, x \in \R^d, x_{d+1} > 0\}$ and solves 
\begin{equation}
\label{eq:lifting} 
  \operatorname*{div} (x_{d+1}^{1-2s} \nabla \lL u)  = 0  \qquad 
\mbox{ in $\R^{d+1}_+$}, 
\qquad \trace \cL u = u. 
\end{equation}
In weak form, the combination of \eqref{eq:CS} and \eqref{eq:lifting} 
therefore yields 
\begin{equation}
\int_{\R^{d+1}_+} x_{d+1}^{1-2s} \nabla \cL u \cdot \nabla \cL v = 
C(d,s) \int_{\R^d} (-\Delta u)^s v \qquad \forall v \in C^\infty_0(\R^{d+1}). 
\end{equation}
For suitable $u$, $v$, we also have 
\begin{align}
\label{eq:def-fractional-a}
\int_{\R^{d+1}_+} x_{d+1}^{1-2s} \nabla \cL u \cdot \nabla \cL v &= 
C(d,s) \int_{\R^d} (-\Delta u)^s v  \\
\label{eq:def-fractional-b}
& = 
\frac{C(d,s)}{2} \int_{\R^d \times \R^d} \frac{(u(x) - u(y))(v(x) - v(y))}{|x - y|^{d+2s}}\,dx\,dy, 
\end{align}
which is a form that is amenable to Galerkin discretizations. 

The fractional Laplacian on a \emph{bounded} domain $\Omega \subset \R^d$ 
can be defined in one of several \emph{non-equivalent} ways. We consider the 
\emph{integral fractional Laplacian} with the exterior ``boundary'' condition 
$u \equiv 0$ in $\Omega^c$, which reads, 
cf., e.g., the discussions in 
\cite{AcostaB_17,ainsworth-et-al18}
\begin{equation}
\label{eq:integral-fractional-laplacian}
(-\Delta u)_I^s(x) = C(d,s) \text{P.V.} \int_{\R^d} \frac{u(x) - u(y)}{|x - y|^{d+2s}}\,dy, 
\qquad x \in \Omega
\end{equation}
and the understanding that $u = 0$ on $\Omega^c$. Important  
for the further developments is that this version of the fractional Laplacian
still admits the interpretation \eqref{eq:CS} as a Dirichlet-to-Neumann map 
for arguments $u \in H^s_0(\R^d;\Omega)$  
(see \eqref{eq:def-Hs0} ahead). 
In particular, for $u$, $v \in H^s_0(\R^d;\Omega)$ 
the representations \eqref{eq:def-fractional-a} and \eqref{eq:def-fractional-b}
are both valid. 

\subsection{Notation}
Let $\R^{d+1}_+ = \R^d\times (0,\infty)$ be the upper half-space. 
We will identify its boundary $\R^{d} \times \{0\}$ with $\R^d$. 
More generally, if necessary, we will identify subsets $\omega \subset \R^d$ 
with $\omega \times \{0\} \subset \R^{d+1}$. 
For measurable subset $M$ of $\R^d$,
we will use standard Lebesgue and Sobolev spaces $L^2(M)$ and $H^1(M)$.
Sobolev spaces of non-integer order $s\in(0,1)$ are defined via the Sobolev-Slobodecki norms
\begin{align*}
  \norm{u}{H^s(M)}^2 = \norm{u}{L^2(M)}^2 + \snorm{u}{H^s(M)}^2 = 
  \norm{u}{L^2(M)}^2 + \int_{M}\int_{M} \frac{\abs{u(x)-u(y)}^2}{\abs{x-y}^{d+2s}}\,dxdy.
\end{align*}
We will work in particular with the Hilbert space
\begin{align}
\label{eq:def-Hs0}
  H^s_0(\R^d;\Omega) := \left\{ u\in H^s(\R^d)\mid u\equiv 0 \text{ on } \R^d\setminus\overline\Omega \right\}.
\end{align}
\section{Main results}
\subsection{Model problem and discretization}
For a polyhedral Lipschitz domain $\Omega\subset\R^d$ 
and $s\in(0,1)$, we are interested in calculating
the trace $u$ on $\Omega \subset \R^d$ of a function $\fu$ defined 
on $\R^{d+1}_+$, where $\fu$ solves 
\begin{align}\label{eq:dirichlet}
  \begin{split}
  -\div \left(  x_{d+1}^{1-2s} \nabla \fu \right) &= 0 \quad\text{ in } \R^{d+1}_+,\\
  - \lim_{x_{d+1} \rightarrow 0+} x_{d+1}^{1-2s} \partial_{x_{d+1}} \fu &= f \quad\text{ on } \Omega,\\
  \fu &= 0 \quad\text{ on } \R^d\setminus\Omega.
  \end{split}
\end{align}
Our variational formulation of~\eqref{eq:dirichlet} is based on the spaces 
$H^s_0(\R^d;\Omega)$: 
Find $u\in H^s_0(\R^d;\Omega)$ such that
\begin{align}\label{eq:dirichlet:var}
  \int_{\R^{d+1}_+} x_{d+1}^{1-2s} \nabla \lL u \cdot \nabla \lL v\,dx = \int_\Omega f v\,dx
  \quad\text{ for all } v\in H^s_0(\R^d;\Omega).
\end{align}
Here, $\lL$ is the \textit{harmonic extension} operator associated with the PDE given in~\eqref{eq:dirichlet}. It has already appeared in 
\eqref{eq:lifting} and is formally defined in \eqref{eq:def-lL}.
We will show in Section~\ref{sec:cont} ahead that the left-hand side of the above equation introduces a bounded and elliptic bilinear form.
Hence, the Lax-Milgram Lemma proves that the variational formulation~\eqref{eq:dirichlet:var} is well-posed.
Given a quasiuniform mesh $\tT_h$ on $\Omega$ with mesh width $h$, we discretize problem \eqref{eq:dirichlet:var} using
the conforming finite element space
\begin{align*}
  \sS^1_0(\tT_h) := \left\{ u\in C(\R^d) \mid 
\operatorname{supp} u \subset \overline\Omega \mbox{ and } 
u|_K \in {\mathcal P}_1\forall K \in \tT_h\right\} \subset H^s_0(\R^d;\Omega),
\end{align*}
where ${\mathcal P}_1$ denote the space of polynomials of degree $1$. 
We emphasize that $\sS^1_0(\tT_h)$ is the ``standard'' space of 
piecewise linear functions on $\Omega$ that are extended by zero outside 
$\Omega$.
Obviously, there is a unique solution $u_h\in \sS^1_0(\tT_h)$ of the linear system
\begin{align}\label{eq:discrete}
  \int_{\R^{d+1}_+} x_{d+1}^{1-2s} \nabla \lL u_h \cdot \nabla \lL v_h\,dx = \int_\Omega f v_h\,dx
  \quad\text{ for all } v_h\in\sS^1_0(\tT_h).
\end{align}
If we consider the nodal basis $(\psi_j)_{j=1}^N$ of $\sS^1_0(\tT_h)$, we can write equation~\eqref{eq:discrete}
as
\begin{align*}
  \Af \xx = \bb.
\end{align*}
Our goal is to derive an $\hH$-matrix representation of the inverse $\Af^{-1}$.
\begin{remark}
Computationally, the bilinear form \eqref{eq:discrete} is 
not easily accessible. One possibility is to 
employ \eqref{eq:def-fractional-b}. For this representation of the bilinear
form, the entries of the stiffness matrix $\Af$ can be computed, 
\cite{AcostaBB_17,ainsworth-glusa17}. 
\eremk
\end{remark}
\subsection{Blockwise low-rank approximation}
Let us introduce the necessary notation.
Let $\Ii = \left\{ 1,\dots,N \right\}$ be the set of indices of the nodal basis $(\psi_j)_{j=1}^N$ of $\sS^1_0(\tT_h)$.
A \textit{cluster} $\tau$ is a subset of $\Ii$. For a cluster $\tau$, we say that $B_{R_\tau}^0\subset\R^d$ is a
\textit{bounding box} if
\begin{itemize}
  \item[(i)] $B_{R_\tau}^0$ is a hyper cube with side length $R_\tau$,
  \item[(ii)] $\supp(\psi_j)\subset B_{R_\tau}^0$ for all $j\in\tau$.
\end{itemize}
For an \textit{admissibility parameter} $\eta>0$, a pair of cluster $(\tau,\sigma)$ is called $\eta$-admissible, if
there exist bounding boxes $B_{R_\tau}^0$ of $\tau$ and $B_{R_\sigma}^0$ of $\sigma$ such that
\begin{align}
\label{eq:admissiblity}
  \max \left\{ \diam(B_{R_\tau}^0), \diam(B_{R_\sigma}^0) \right\} \leq \eta\, \dist \left( B_{R_\tau}^0, B_{R_\sigma}^0 \right).
\end{align}
The next theorem is the first main result of this work. 
For two admissible clusters, the associated matrix block
of the inverse $\Af^{-1}$ of the matrix associated to the linear system of problem~\eqref{eq:dirichlet}
can be approximated by low-rank matrices with an error that is exponentially small in the rank.
\begin{theorem}\label{thm:approx:blocks}
  Let $\eta>0$ be a fixed admissibility parameter and $q\in(0,1)$. Let 
  $(\tau,\sigma)$ be a cluster pair with $\eta$-admissible bounding boxes.
  Then, for each $k\in\N$, there exist matrices $\Xf_{\tau\sigma}\in\R^{\abs{\tau}\times r}$ and $\Yf_{\tau\sigma}\in\R^{\abs{\sigma}\times r}$
  with rank $r\leq C_{\rm dim}(2+\eta)^{d+1}q^{-(d+1)}k^{d+2}$ such that
  \begin{align}\label{thm:approx:blocks:eq}
    \norm{\Af^{-1}|_{\tau\times\sigma} - \Xf_{\tau\sigma}\Yf_{\tau\sigma}^{\top}}{2} \leq C_{\rm apx} N^{\frac{1+d}{d}} q^k.
  \end{align}
  The constants $C_{\rm dim}$ and $C_{\rm apx}$ depend only on $d$, $\Omega$, the shape-regularity of $\tT_h$, and on $s$.
\end{theorem}
Theorem~\ref{thm:approx:blocks} shows that individual blocks of 
$\Af^{-1}$ can be approximated by low-rank matrices. $\calH$-matrices 
are blockwise low-rank matrices where the blocks are organized 
in a tree structure, which affords the fast arithmetic of $\calH$-matrices. 
The block cluster tree is based on a tree structure for the index set 
$\Ii$, which we described next.  
\begin{definition}[cluster tree]
A \emph{cluster tree} with \emph{leaf size} $n_{\rm leaf} \in \mathbb{N}$ is a binary tree $\mathbb{T}_{\Ii}$ with root $\mathcal{I}$
such that for each cluster $\tau \in \mathbb{T}_{\Ii}$ the following dichotomy holds: either $\tau$ is a leaf of the tree and
$\abs{\tau} \leq n_{\rm leaf}$, or there exist so called sons $\tau'$, $\tau'' \in \mathbb{T}_{\Ii}$, which are disjoint subsets of $\tau$ with
$\tau = \tau' \cup \tau''$. The \emph{level function} ${\rm level}: \mathbb{T}_{\Ii} \rightarrow \mathbb{N}_0$ is inductively defined by
${\rm level}(\Ii) = 0$ and ${\rm level}(\tau') := {\rm level}(\tau) + 1$ for $\tau'$ a son of $\tau$. The \emph{depth} of a cluster tree
is ${\rm depth}(\mathbb{T}_{\Ii}) := \max_{\tau \in \mathbb{T}_{\Ii}}{\rm level}(\tau)$.
\end{definition}

\begin{definition}[far field, near field, and sparsity constant]
A partition $P$ of $\Ii \times \Ii$ is said to be based on the cluster tree $\mathbb{T}_{\Ii}$,
if $P \subset \mathbb{T}_{\Ii}\times\mathbb{T}_{\Ii}$. For such a partition $P$ and fixed
admissibility parameter $\eta > 0$,
we define the \emph{far field} and the \emph{near field} as
\begin{equation}\label{eq:farfield}
P_{\rm far} := \{(\tau,\sigma) \in P \; : \; (\tau,\sigma) \; \text{is $\eta$-admissible}\}, \quad P_{\rm near} := P\backslash P_{\rm far}.
\end{equation}
The \emph{sparsity constant} $C_{\rm sp}$,
introduced in \cite{HackbuschKhoromskij2000a,HackbuschKhoromskij2000b,GrasedyckDissertation}, of such a partition is defined by
\begin{equation}\label{eq:sparsityConstant}
C_{\rm sp} := \max\left\{\max_{\tau \in \mathbb{T}_{\Ii}}\abs{\{\sigma \in \mathbb{T}_{\Ii} \, : \, \tau \times \sigma \in P_{\rm far}\}},
\max_{\sigma \in \mathbb{T}_{\Ii}}\abs{\{\tau \in \mathbb{T}_{\Ii} \, : \, \tau \times \sigma \in P_{\rm far}\}}\right\}.
\end{equation}
\end{definition}

The following Theorem~\ref{th:Happrox} shows that the matrix
$\Af^{-1}$ can be approximated by blockwise rank-$r$ matrices at an exponential rate in the
block rank $r$:

\begin{theorem}\label{th:Happrox}
Fix the admissibility parameter $\eta > 0$.
Let a partition $P$ of ${\Ii} \times {\Ii}$ be based on a cluster tree
$\mathbb{T}_{\Ii}$. Then, there is a
blockwise rank-$r$ matrix $\mathbf{B}_{\mathcal{H}}$ such that
\begin{equation}
\|\Af^{-1} - \mathbf{B}_{\mathcal{H}}\|_2 \leq C_{\rm apx} C_{\rm sp} N^{(d+1)/d} \operatorname*{depth} (\mathbb{T}_{\Ii}) e^{-br^{1/(d+2)}}.
\end{equation}
The constant $C_{\rm apx}$ depends only on $\Omega$, $d$,
the shape regularity of the quasiuniform triangulation $\mathcal{T}_h$, and on s,
while the constant $b>0$ additionally depends on $\eta$.
\end{theorem}
\begin{proof}
As it is shown in 
\cite{GrasedyckDissertation}, \cite[Lemma~{6.32}]{hackbusch15}, norm bounds
for a block matrix that is based on a cluster tree can be inferred from
norm bounds for the blocks. This allows one to prove Theorem~\ref{th:Happrox}
based on the results of Theorem~\ref{thm:approx:blocks} 
(see, e.g., the proof of \cite[Thm.~{2}]{FMP_NumerMath_15} for details).
\end{proof}
\begin{remark}
For quasiuniform meshes with $\mathcal{O}(N)$ elements,
typical clustering strategies such as the ``geometric clustering'' described in \cite{hackbusch15}
lead to fairly balanced cluster trees $\mathbb{T}_{\Ii}$ with 
$\operatorname*{depth} \mathbb{T}_{\Ii} = \mathcal{O}(\log N)$ and a sparsity constant $C_{\rm sp}$ 
that is bounded uniformly in $N$.
We refer to~\cite{HackbuschKhoromskij2000a,HackbuschKhoromskij2000b,GrasedyckDissertation,hackbusch15} for the fact that the
memory requirement to store $\mathbf{B}_{\mathcal{H}}$ is $\mathcal{O}\big((r+n_{\rm leaf}) N \log N\big)$.
\eremk
\end{remark}
%
\section{The Beppo-Levi space $\BL^1_\alpha(\R^{d+1}_+$)}\label{sec:cont}
In the present section, we formulate a functional framework for the lifting operator 
$\cL$ of \eqref{eq:lifting}. We will work in the Beppo-Levi space
\begin{align*}
  \BL^1_\alpha(\R^{d+1}_+)
  := \left\{ \fu \in \dD'(\R^{d+1}_+) \mid \nabla \fu \in L^2_\alpha(\R^{d+1}_+) \right\}
\end{align*}
of all distributions $\dD'(\R^{d+1}_+)$ having all first order partial
derivatives in $L^2_\alpha(\R^{d+1}_+)$ for 
\begin{equation}
\label{eq:alpha} \alpha=1-2s \in (-1,1),
\end{equation}
where this last space is defined as the set
of measurable functions $\fu$ such that
\begin{align*}
  \norm{\fu}{L^2_\alpha(\R^{d+1}_+)}^2 = \int_{\R^{d+1}_+} x_{d+1}^\alpha\abs{\fu(x)}^2\,dx < \infty.
\end{align*}
We denote by $L_{\alpha,\rm{bdd}}^2(\R^{d+1}_+)$ the set of functions that are in $L_\alpha^2$
on every bounded subset of $\R^{d+1}_+$.
By $\trace: C^{\infty}(\overline{\R^{d+1}_+})\rightarrow C^{\infty}(\R^d)$ we denote the trace operator
$(\trace\,\fu)(x_1, \dots, x_d) := \fu(x_1, \dots, x_d, 0)$.
The following result, which is an extension to weighted spaces of 
the well-known result~\cite[Cor.~{2.1}]{DenyL_55}, shows that the distributions in $\BL^1_\alpha(\R^{d+1}_+)$
are actually functions. 
Its proof will be given below in Section~\ref{section:tech}.
\begin{lemma}\label{thm:bl}
  For $\alpha\in(-1,1)$ there holds
  $\BL^1_\alpha(\R^{d+1}_+) \subset L_{\alpha,\rm{bdd}}^2(\R^{d+1}_+)$.
Furthermore, for $\alpha \in [0,1)$ one has $\fu \in L^2_{0,{\rm bdd}}(\R^{d+1}_+)$. 
\end{lemma}
We additionally define the space
\begin{align*}
  \BL^s(\R^d) := \left\{ u\in L^2_{\rm loc}(\R^d) \mid \snorm{u}{H^s(\R^d)}<\infty \right\}.
\end{align*}
From now on, we fix a hypercube $K := K'\times (0,b_{d+1})$, $K' = \prod_{j=1}^d (a_j,b_j)$.
Then, using Lemma~\ref{thm:bl}, one can show that $\BL^1_\alpha(\R^{d+1}_+)$ 
and $\BL^s(\R^d)$ are Hilbert spaces when endowed with the norms
\begin{align}
\label{eq:norm-BL1}
  \norm{\fu}{\BL^1_{\alpha}(\R^{d+1}_+)}^2 := \norm{\fu}{L^2_\alpha(K)}^2 + \norm{\nabla\fu}{L^2_\alpha(\R^{d+1}_+)}^2 \quad\text{ and }\quad
  \norm{u}{\BL^s(\R^d)}^2 := \norm{u}{L^2(K')}^2 + \snorm{u}{H^s(\R^d)}^2.
\end{align}
There holds the following density result, which can be found for bounded domains in~\cite[Thm.~{11.11}]{Kufner_85} even for higher Sobolev regularity.
In the present case of first order regularity and unbounded domains, 
we give a short proof below in Section~\ref{section:tech}.
\begin{lemma}\label{lem:bl:dense}
  For $\alpha\in(-1,1)$ the set 
  $C^\infty(\overline{\R^{d+1}_+})\cap \BL^1_\alpha(\R^{d+1}_+)$ is dense in $\BL^1_\alpha(\R^{d+1}_+)$.
\end{lemma}
The trace operator can be extended to the spaces $\BL^1_\alpha(\R^{d+1}_+)$ as will also be shown below in Section~\ref{section:tech}.
Analogous trace theorems in Sobolev spaces on smooth and bounded domains
are given for $s=1/2$ in~\cite[Prop.~{1.8}]{CabreT_10}, and for $s\in(0,1)\setminus\frac{1}{2}$ in~\cite[Prop.~{2.1}]{CapellaDDS_11}.
\begin{lemma}\label{lem:bl:trace}
  For $\alpha\in(-1,1)$, the trace operator is a bounded linear operator $\trace: \BL^1_\alpha(\R^{d+1}_+) \rightarrow \BL^s(\R^d)$, where $s$ is given by \eqref{eq:alpha}.
\end{lemma}
We define the Hilbert space $\BL^1_{\alpha,0}(\R^{d+1}_+) := \ker(\trace)$.
The following Poincar\'e inequality holds on this space. The proof will be given below
in Section~\ref{section:tech}.
\begin{cor}\label{cor:poincare}
  For all $\fu\in\BL^1_{\alpha,0}(\R^{d+1}_+)$, there holds
  \begin{align*}
    \norm{\fu}{\BL^1_{\alpha}(\R^{d+1}_+)} \lesssim \norm{\nabla \fu}{L^2_\alpha(\R^{d+1}_+)}.
  \end{align*}
\end{cor}
For a function $u\in H^s(\R^d)$, we define the \textit{minimum norm extension} or \textit{harmonic extension} $\lL u\in \BL^1_\alpha(\R^{d+1}_+)$ as
\begin{align}
\label{eq:def-lL}
  \lL u = \argmin_{\substack{ \fu \in \BL^1_\alpha(\R^{d+1}_+)\\ \trace\,\fu = u}} \norm{\nabla \fu}{L^2_\alpha(\R^{d+1}_+)}.
\end{align}
We can characterize $\lL u$ by
\begin{align}\label{eq:minnorm:var}
  \begin{split}
  \int_{\R^{d+1}_+}x_{d+1}^{1-2s} \nabla \lL u \cdot \nabla \vV\,dx &= 0
  \quad\text{ for all } \vV \in \BL^1_{\alpha,0}(\R^{d+1}_+),\\
  \trace \lL u &= u.
  \end{split}
\end{align}
In view of the previous developments the minimum norm extension exists 
uniquely and satisfies 
\begin{align}\label{eq:1}
  \norm{\lL u}{\BL^1_\alpha(\R^{d+1}_+)} \lesssim \norm{u}{H^s(\R^d)}.
\end{align}
Indeed, the minimum norm extension can be written $\lL u = \eE u + \fu$, where $\eE u$ is the operator
from Lemma~\ref{lem:ext}, and $\fu\in \BL^1_{\alpha,0}(\R^{d+1}_+)$ is given by
\begin{align*}
  \int_{\R^{d+1}_+}x_{d+1}^{1-2s} \nabla \fu \cdot \nabla \vV\,dx = -\int_{\R^{d+1}_+}x_{d+1}^{1-2s} \nabla \eE u \cdot \nabla \vV\,dx
  \quad\text{ for all } \vV \in \BL^1_{\alpha,0}(\R^{d+1}_+).
\end{align*}
This equation is uniquely solvable due to the Lax-Milgram theorem and Corollary~\ref{cor:poincare},
and this also implies the stability~\eqref{eq:1}.
Due to~\eqref{eq:minnorm:var}, we see that a variational form of our
original problem~\eqref{eq:dirichlet} is actually given by~\eqref{eq:dirichlet:var}.
Next, we show that problem~\eqref{eq:dirichlet:var} is well-posed. We mention that ellipticity has already been
shown in~\cite[eq.~(3.7)]{CaffarelliS_07} using Fourier methods.
\begin{lemma}\label{lem:dirichlet:var}
  Problem~\eqref{eq:dirichlet:var} has a unique solution $u\in H^s_0(\R^d;\Omega)$, and
  \begin{align*}
    \norm{u}{H^s(\R^d)} \lesssim \norm{f}{H^{-s}(\Omega)},
  \end{align*}
  where $H^{-s}(\Omega)$ is the dual space of $H^s_0(\R^d;\Omega)$.
\end{lemma}
\begin{proof}
  Due to~\cite[Prop.~2.4]{AcostaB_17}, there holds the Poincar\'e inequality
  $\norm{u}{L^2(\Omega)}\lesssim \snorm{u}{H^s(\R^d)}$ for all $u\in H^s_0(\R^d;\Omega)$.
  We conclude that $\norm{u}{H^s(\R^d)}\lesssim \snorm{u}{H^s(\R^d)}$ for all $u\in H^s_0(\R^d;\Omega)$.
  Combining this Poincar\'e inequality with the trace estimate~\eqref{lem:trace:eq1}, we obtain
  the ellipticity of the bilinear form on the left-hand side of~\eqref{eq:dirichlet:var}.
  The continuity of this bilinear form follows from~\eqref{eq:1}.
\end{proof}
\subsection{Technical details and proofs}\label{section:tech}
Define the Sobolev space $H^1_\alpha(\R^{d+1}_+)$ as the space of functions $\fu$ such that
\begin{align*}
  \norm{\fu}{L^2_\alpha(\R^{d+1}_+)}^2 + \norm{\nabla\fu}{L^2_\alpha(\R^{d+1}_+)}^2 < \infty.
\end{align*}
We start with a density result, whose proof is based on ideas from~\cite[Thm.~11.11]{Kufner_85}.
\begin{lemma}\label{lem:density}
  For $\alpha\in(-1,1)$, the space $C^{\infty}(\overline{\R^{d+1}_+})\cap H^1_\alpha(\R^{d+1}_+)$ is dense in
  $H^1_\alpha(\R^{d+1}_+)$.
\end{lemma}
\begin{proof}
By \cite[Thm.~1]{GoldshteinU_09}
  the space $C^\infty(\R^{d+1}_+)\cap H^1_\alpha(\R^{d+1}_+)$ 
is dense in $H^1_\alpha(\R^{d+1}_+)$. 
  Hence, without loss of generality, we may assume that $\fu\in C^\infty(\R^{d+1}_+)\cap H^1_\alpha(\R^{d+1}_+)$.
     For $h>0$, define the function $\fu_h$ by
      \begin{align*}
        \fu_h(x_1,\dots,x_{d+1}):=
        \begin{cases}
          \fu(x_1,\dots,x_{d+1}) & \text{ if } h<x_{d+1}\\
          \fu(x_1,\dots,h) & \text{ if } x_{d+1}\leq h.
        \end{cases}
      \end{align*}
      By construction, $\fu_h\in C(\overline{\R^{d+1}_+})\cap H^1_\alpha(\R^{d+1}_+)$ and
      \begin{align}\label{lem:density:eq2}
        \begin{split}
          \norm{\fu-\fu_h}{H^1_\alpha(\R^{d+1}_+)}^2 = \norm{\fu-\fu_h}{H^1_\alpha(\R^d\times(0,h))}^2
          &= \norm{\fu-\fu_h}{L^2_\alpha(\R^d\times(0,h))}^2 + \norm{\nabla \fu}{L^2_\alpha(\R^d\times(0,h))}^2\\
          &\lesssim \norm{\fu}{H^1_\alpha(\R^d\times(0,h))}^2 + \norm{\fu_h}{L^2_\alpha(\R^d\times(0,h))}^2.
        \end{split}
      \end{align}
By Lebesgue Dominated Convergence, we have 
$\lim_{h \rightarrow 0} \norm{\fu}{H^1_\alpha(\R^d\times (0,h))} = 0$. Hence, we focus on showing
$\lim_{h \rightarrow 0} \norm{\fu_h}{L^2_\alpha(\R^d\times (0,h))} = 0$. To that end, we use a 1D trace
inequality: For $v \in C^1(0,\infty)$ we have $v(h) = v(y) - \int_{h}^y v^\prime(t)\,dt$ so that 
\begin{align*}
\int_{y=h}^{h+1} y^\alpha v^2(h)\,dy &\leq 
2 \int_{y=h}^{h+1} y^\alpha v^2(y) \,dy 
+ 2 \int_{y=h}^{h+1} y^\alpha \left| \int_{t=h}^y |v^\prime(t)|\,dt\right|^2\,dy\\
&\lesssim \|v\|^2_{L^2_\alpha(h,h+1)} + \int_{y=h}^{h+1} y^\alpha y^{1-\alpha} \int_{t=h}^{h+1} t^\alpha |v^\prime(t)|^2\,dt\,dy
\lesssim \|v\|^2_{L^2_\alpha(h,h+1)} + \|v^\prime\|^2_{L^2_\alpha(h,h+1)}. 
\end{align*}
Since there exists $C > 0$ such that  for $h \in (0,1]$, 
we have $C^{-1} \leq \int_{h}^{h+1} t^\alpha\,dt \leq C$, we can conclude 
\begin{align}
\label{eq:1d-trace}
|v(h)|^2 \leq C^2_{\rm trace}  \left[ \|v\|^2_{L^2_\alpha(h,h+1)} + \|v^\prime\|^2_{L^2_\alpha(h,h+1)}
\right].
\end{align}
With this, we estimate 
      \begin{align*}
        \norm{\fu_h}{L^2_\alpha(\R^d\times(0,h))}^2 &= \int_0^h x_{d+1}^\alpha\int_{y\in\R^d}\fu_h(y,h)^2\, dy\, dx_{d+1}
        = h^{\alpha+1} \int_{y\in\R^d}\fu(y,h)^2\, dy\\
        &\stackrel{\eqref{eq:1d-trace}}{\leq} C^2_{\rm trace} h^{\alpha+1} \norm{\fu}{H^1_\alpha(\R^{d}\times(h,h+1))}^2
        \leq C^2_{\rm trace} h^{\alpha+1} \norm{\fu}{H^1_\alpha(\R^{d+1}_{+})}^2.
      \end{align*}
      As $\alpha+1>0$ we conclude that $\lim_{h \rightarrow} \norm{\fu_h}{L^2_\alpha(\R^d\times (0,h))} = 0$. 
      Since $\fu_h$ is only piecewise smooth, we perform, as a last step, a mollification step.   
      The above shows that, given $\varepsilon>0$, we can fix $h$ such that
      \begin{align}\label{lem:density:eq3}
        \norm{\fu-\fu_h}{H^1_\alpha(\R^{d+1}_+)} \leq \varepsilon.
      \end{align}
      Next, for $0<\delta<h/4$ define the function
      \begin{align*}
        \wilde \fu_\delta :=
        \begin{cases}
          \rho_{\delta} \star \fu_h & h/2 < x_{d+1}\\
          \fu_h & x_{d+1}\leq h/2.
        \end{cases}
      \end{align*}
      Then, $\wilde \fu_\delta\in C^\infty(\overline{\R^{d+1}_+})\cap H^1_\alpha(\R^{d+1}_+)$, cf.~\cite{GoldshteinU_09}, and
      \begin{align}\label{lem:density:eq4}
        \norm{\fu_h-\wilde \fu_\delta}{H^1_\alpha(\R^{d+1}_+)} = \norm{\fu_h-\rho_\delta \star \fu_h}{H^1_\alpha(\R^d\times(h/2,\infty))}.
      \end{align}
      Note that $h$ is already fixed. Standard results about mollification, cf., e.g.,~\cite{GoldshteinU_09},
      show that the term $\norm{\fu_h-\rho_\delta \star \fu_h}{H^1_\alpha(\R^d\times(h/2,\infty))}$
      converges to zero for $\delta\rightarrow 0$. Hence, choosing $\delta$ small enough, we obtain from~\eqref{lem:density:eq3}
      and~\eqref{lem:density:eq4} that $\norm{\fu-\wilde \fu_\delta}{H^1_\alpha(\R^{d+1}_+)}\leq 2\varepsilon$, which proves the result.
\end{proof}
Next, we show that the trace operator $\trace$ extends continuously to weighted Sobolev spaces.
\begin{lemma}\label{lem:mult:trace}
  Let $\alpha\in(-1,1)$. The trace operator $\trace$ has a unique extension
  as linear and bounded operator $H^1_\alpha(\R^{d+1}_+)\rightarrow L^2(\R^d)$,
  and there holds the multiplicative trace inequality
  \begin{align}\label{lem:mult:trace:eq}
    \norm{\trace\, \fu}{L^2(\Omega)} \leq
    C_{\rm tr} \norm{\fu}{L^2_\alpha(\Omega_+)}^{(1-\alpha)/2} \cdot
    \norm{\partial_{d+1} \fu}{L^2_\alpha(\Omega_+)}^{(1+\alpha)/2},
  \end{align}
  for all measurable subsets $\Omega\subseteq\R^d$, where $\Omega_+ := \Omega\times (0,\infty)$.
  The constant $C_{\rm tr}$ does not depend on $\Omega$.
\end{lemma}
\begin{proof}
  In order to prove all statements of the lemma, we note that due to Lemma~\ref{lem:density}, it
  is sufficient to show the estimate~\eqref{lem:mult:trace:eq} for smooth functions
  $\fu\in C^\infty(\overline{\R^{d+1}_+})\cap H^1_\alpha(\R^{d+1}_+)$. 
 We may also assume that $\fu$ is supported by $\R^d \times (0,1)$.
  Using the abreviation $v(x) = \fu(x_1,\dots,x_d,x)$, we note that due to H\"older's inequality
  \begin{align*}
    \abs{v(0)} &\leq \abs{v(y)} + \abs{\int_0^y v'(t)\,dt}
    \lesssim \abs{v(y)} + y^{(1-\alpha)/2} \norm{v'}{L^2_\alpha(\R_+)}.
  \end{align*}
  A one-dimensional trace inequality and a scaling argument show for $y > 0$
  \begin{align*}
    \abs{v(y)}^2
    &\lesssim y^{-1} \int_y^{2y} \abs{v(t)}^2\,dt + y \int_y^{2y} \abs{v'(t)}^2\,dt.
  \end{align*}
  For $t\in(y,2y)$ we have $1\leq y^{-\alpha} t^\alpha$ if $\alpha\in [0,1)$
  and $1 \leq 2^{-\alpha} y^{-\alpha} t^\alpha$ if $\alpha \in (-1,0)$,
  and we conclude
  \begin{align*}
    \abs{v(y)}^2 &\lesssim y^{-1-\alpha}\norm{v}{L^2_\alpha(\R_+)}^2 + y^{1-\alpha}\norm{v'}{L^2_\alpha(\R_+)}^2.
  \end{align*}
  If $\norm{v'}{L^2_\alpha(\R_+)} \ne 0$, we set 
  $y = \norm{v}{L^2_\alpha(\R_+)} \cdot \norm{v'}{L^2_\alpha(\R_+)}^{-1}$ and 
  get 
  \begin{align}
  \label{eq:1D-trace}
    \abs{v(0)}^2 \lesssim \norm{v}{L^2_\alpha(\R_+)}^{1-\alpha}\cdot\norm{v'}{L^2_\alpha(\R_+)}^{1+\alpha}.
  \end{align}
  We note that \eqref{eq:1D-trace} is also valid if 
  $\norm{v'}{L^2_\alpha(\R_+)} = 0$ since our assumption
  $\operatorname{supp} \fu \in \R^d \times [0,1]$ implies in 
  this degenerate case $v \equiv 0$. 
  Integrating $\fu(\cdot,0)$ over $\Omega$ and using \eqref{eq:1D-trace} 
  shows~\eqref{lem:mult:trace:eq}.
\end{proof}
\begin{lemma}\label{lem:trace}
  Let $\alpha\in(-1,1)$ and $s$ be given by \eqref{eq:alpha}.
  The trace operator $\trace$ is bounded as
  $\trace: H^1_\alpha(\R^{d+1}_+)\rightarrow H^s(\R^d)$, and
  \begin{align}\label{lem:trace:eq1} 
    \snorm{\trace\,\fu}{H^s(\R^d)} \lesssim \norm{\nabla \fu}{L^2_\alpha(\R^{d+1}_+)}. 
  \end{align}
\end{lemma}
\begin{proof}
  Due to Lemma~\ref{lem:density}, it suffices to show~\eqref{lem:trace:eq1} for $\fu\in C^\infty(\overline{\R^{d+1}_+})\cap H^1_\alpha(\R^{d+1}_+)$. 
  Combining~\eqref{lem:trace:eq1} with Lemma~\ref{lem:mult:trace} then proves that $\trace: H^1_\alpha(\R^{d+1}_+)\rightarrow H^s(\R^d)$ is bounded.

Upon writing $y = x + r \phi$ with polar coordinates 
  $r > 0$, $\phi\in S^{d-1}:=\partial B_1(0) \subset \R^d$, we obtain 
  with the triangle inequality and symmetry arguments 
  \begin{align*}
    \snorm{\trace\,\fu}{H^s(\R^d)}^2 &= \int_{\R^d}\int_{\R^d} \frac{\abs{\fu(x,0)-\fu(y,0)}^2}{\abs{x-y}^{d+2s}} \,dy\,dx
    \lesssim \int_{\R^d}\int_{\R^d} \frac{\abs{\fu(\tfrac{x+y}{2},\tfrac{\abs{x-y}}{2})-\fu(x,0)}^2}{\abs{x-y}^{d+2s}} \,dy\,dx\\
    &\sim \int_x \int_{\phi \in S^{d-1}}\int_{r=0}^\infty \frac{\abs{\fu(x + \tfrac{r}{2}\phi,\tfrac{r}{2})-\fu(x,0)}^2}{r^{1+2s}}\,dr\,d\phi\,dx. 
  \end{align*}
The fundamental theorem of calculus gives 
  \begin{align*}
    \fu(x + \tfrac{r}{2}\phi,\tfrac{r}{2})-\fu(x,0) = 
\int_0^{r} \nabla_{x} \fu(x + \tfrac{y}{2} \phi,\tfrac{y}{2})\cdot \phi +
    \partial_{d+1}\fu(x+\tfrac{y}{2}\phi,\tfrac{y}{2})\,dy,
  \end{align*}
  and the weighted Hardy inequality from~\cite[I, Thm.~9.16]{Zygmund_02}, cf. also~\cite[(1.1)]{Muckenhoupt_72}, then implies 
  \begin{align*}
    \int_{r=0}^\infty \frac{\abs{\fu(x+\tfrac{r}{2}\phi,\tfrac{r}{2})-\fu(x,0)}^2}{r^{1+2s}}\,dr \lesssim
    \int_{r=0}^\infty r^{1-2s} \abs{\nabla_{x} \fu(x+\tfrac{r}{2}\phi,\tfrac{r}{2}) \cdot \phi+ \partial_{d+1}\fu(x+\tfrac{r}{2}\phi,\tfrac{r}{2})}^2.
  \end{align*}
  Hence,
  \begin{align*}
    \snorm{\trace\,\fu}{H^s(\R^d)}^2 \lesssim
    \int_{x\in\R^d}\int_{\phi \in S^{d-1}} \int_{r=0}^\infty r^{1-2s} 
\abs{\nabla \fu(x+\tfrac{r}{2}\phi,\tfrac{r}{2})}^2\,dr\,d\phi\,dx
    \lesssim \norm{\nabla \fu}{L^2_\alpha(\R^{d+1}_+)}^2,
  \end{align*}
  which proves~\eqref{lem:trace:eq1}.
\end{proof}
Next, we will show that the trace operator $\trace: H^1_\alpha(\R^{d+1}_+)\rightarrow H^s(\R^d)$
is actually onto. To that end, we generalize ideas from~\cite{Gagliardo_57}.
\begin{lemma}\label{lem:ext}
  Let $\alpha\in(-1,1)$ and $s$ be given by \eqref{eq:alpha}.
  There exists a bounded linear operator
  $\eE:H^s(\R^d)\rightarrow H^1_\alpha(\R^{d+1}_+)$ that is a right-inverse of the trace operator $\trace$.
  Furthermore, there exists a constant $C>0$ such that for all $h>0$ it holds
  \begin{align*}
    \norm{\eE u}{L^2_\alpha(\R^d\times (0,h))} \leq C h^{1-s}\norm{u}{L^2(\R^d)}.
  \end{align*}
\end{lemma}
\begin{proof}
  Let $\rho\in C^\infty_0(\R^d)$ and $\eta\in C^\infty(\R)$ 
  with $\operatorname*{supp} \eta \subset (-1,1)$ and $\eta\equiv 1$ in $(-1/2,1/2)$.
  We denote a point in $\R^{d+1}_+$ by $(x,t)$ with $x\in\R^{d}$.
  Define the extension operator as the mollification $\eE u (x,t) := \eta(t) \rho_t \star u(x)$, where
  $\rho_t(y) := t^{-d} \rho(y/t)$. Since 
  $\norm{\rho_t\star u}{L^2(\R^d)}\lesssim \norm{u}{L^2(\R^d)}$ uniformly in $t > 0$ (cf., e.g., 
  \cite[Thm.~{2.29}]{adams-fournier03}), 
we immediately obtain the postulated estimate
  \begin{align*}
    \norm{\eE u}{L^2_\alpha(\R^d\times (0,h))}^2 \leq \norm{\eta}{L^{\infty}}^2
    \int_0^h t^\alpha \norm{\rho_t\star u}{L^2(\R^d)}^2\,dt
    \lesssim
    h^{2(1-s)}
    \norm{u}{L^2(\R^d)}^2.
  \end{align*}
  Since $\eta$ is compactly supported, this also shows $\norm{\eE u}{L^2_\alpha(\R^{d+1}_+)} \lesssim \norm{u}{L^2(\R^d)}$.
  For the desired statement that $\eE:H^s(\R^d)\rightarrow H^1_\alpha(\R^{d+1}_+)$ is bounded it is sufficient to prove
  \begin{align*}
    \norm{\nabla_{x}(\rho_t\star u)}{L^2_\alpha(\R^{d+1}_+)}^2 + \norm{\partial_t (\rho_t\star u)}{L^2_\alpha(\R^{d+1}_+)}^2
    \lesssim \snorm{u}{H^s(\R^d)}^2. 
  \end{align*}
  To that end, we calculate
  \begin{align*}
    \partial_t (\rho_t\star u)(x)
    &= -d t^{-d-1}\int_{\R^d} u(y) \rho\left( \frac{x-y}{t} \right)\,dy
    - t^{-d-2} \int_{\R^d} u(y) \nabla\rho \left( \frac{x-y}{t} \right)\cdot(x-y)\,dy.
  \end{align*}
  Integration by parts shows $\int_{\R^d} \nabla \rho(z)\cdot z\,dz = -d$, which yields
  \begin{align*}
    -d t^{-d-1}\int_{\R^d} \rho\left( \frac{x-y}{t} \right)\,dy
    = -d t^{-1}
    &= t^{-d-2} \int_{\R^d} \nabla\rho \left( \frac{x-y}{t} \right) \cdot (x-y)\,dy.
  \end{align*}
  Hence, we can write
  \begin{align*}
    \begin{split}
    \partial_t (\rho_t \star u)(x)
      &= -d t^{-d-1}\int_{\R^d} [u(y)-u(x)] \rho\left( \frac{x-y}{t} \right)\,dy\\
      &\qquad- t^{-d-2} \int_{\R^d} [u(y)-u(x)] \nabla\rho \left( \frac{x-y}{t} \right)\cdot(x-y)\,dy.
    \end{split}
  \end{align*}
  Next, we calculate for $1\leq j\leq d$
  \begin{align*}
    \partial_{x_j} (\rho_t\star u)(x) = t^{-d-1} \int_{\R^d} u(y) \left( \partial_{x_j} \rho \right)\left( \frac{x-y}{t} \right)\,dy.
  \end{align*}
  Integration by parts also shows that $\int (\partial_{x_j}\rho)(z)\,dz=0$, which yields
  \begin{align*}
    \partial_{x_j} (\rho_t\star u)(x) = t^{-d-1} \int_{\R^d} [u(y)-u(x)] \left( \partial_{x_j} \rho \right)\left( \frac{x-y}{t} \right)\,dy. 
  \end{align*}
  Due to the support properties of $\rho$, we conclude
  \begin{align*}
    \abs{\partial_t (\rho_t\star u)(x)} + \abs{\nabla_x (\rho_t\star u)(x)} \lesssim t^{-d-1} \int_{B_t(x)} \abs{u(x)-u(y)}\,dy,
  \end{align*}
  where $B_r(x)\subset \R^d$ denotes the ball of radius $t$ centered at $x$.
  Using polar coordinates and Hardy's inequality gives 
  \begin{align*}
    & \int_{0}^\infty t^\alpha \left( \abs{\partial_t (\rho_t\star u)(x)}^2 + \abs{\nabla_x (\rho_t\star u)(x)}^2 \right)
    \,dt
    \lesssim \int_{0}^\infty t^\alpha \left( t^{-d-1}\int_{B_t(x)} \abs{u(y)-u(x)}\,dy
    \right)^2\,dt\\
    &\quad \leq \int_0^\infty \left( t^{-1}
    \int_{B_t(x)} \frac{\abs{u(y)-u(x)}}{\abs{x-y}^{d-\alpha/2}}\,dy  \right)^2\,dt
= \int_0^\infty \left( t^{-1}
    \int_{B_t(0)} \frac{\abs{u(x)-u(x-z)}}{\abs{z}^{d-\alpha/2}}\,dz \right)^2\,dt
\\ &\quad 
= \int_0^\infty \left( t^{-1}
      \int_{r=0}^t \int_{\phi \in S^{s-1}} \frac{\abs{u(x)-u(x-r\phi))}}{r^{1-\alpha/2}}\,d\phi\,dr \right)^2\,dt
\\ &
\quad \lesssim 
    \int_{t=0}^\infty \left( \int_{\phi\in S^{d-1}} \frac{\abs{u(x)-u(x-t\phi))}}{t^{1-\alpha/2}}\,d\phi\right)^2\,dt
    \leq  \int_{\R^d} \frac{\abs{u(x)-u(y)}^2}{\abs{x-y}^{d+2s}}\,dy.
  \end{align*}
  Integrating this estimate over $x\in\R^d$ concludes the proof.
\end{proof}
We are in position to prove Lemma~\ref{thm:bl}.
\begin{proof}[Proof of Lemma~\ref{thm:bl}]
The proof follows a standard procedure. Since it involves functions 
in a half-space, we present some details. 

\emph{Step 1:} 
Let $\rho \in C^\infty_0(\R^{d+1})$ be a symmetric, non-negative function
with $\operatorname{supp} \rho \subset B_1(0)$ and 
set $\rho_\varepsilon(x):= \varepsilon^{-d} \rho(x/\varepsilon)$. 
Introduce the translation operator $\tau_h$ by $\tau_h \varphi(x):= 
\varphi(x - h e_{d+1})$ with $e_{d+1} = (0,0,\ldots,0,1) \in \R^{d+1}$. 
Define for $\varepsilon > 0$ the smoothing operator $\cA_\varepsilon$ by 
$
\cA_\varepsilon \varphi = \rho_\varepsilon \star (\tau_{2\varepsilon} \varphi)
$
and the regularized distribution 
$\fu_\varepsilon$ by 
$$
\langle \fu_\varepsilon,\varphi\rangle:= 
\langle \fu, \cA_\varepsilon \varphi\rangle = 
\langle \fu, \rho_\varepsilon \star (\tau_{2\varepsilon} \varphi)\rangle, 
$$
where we view $\varphi \in \dD(\R^{d+1}_+)$ as an element of 
$\varphi \in \dD(\R^{d+1})$ in the canonical way. 
Note that 
$\fu_\varepsilon  \in C^\infty(\R^{d+1}_+)$ by standard arguments and 
$\operatorname{supp} \fu_\varepsilon \subset 
\R^{d} \times (\varepsilon,\infty)$. We also note that 
\begin{equation}
\lim_{\varepsilon \rightarrow 0}
\langle \fu_\varepsilon ,\varphi\rangle 
= \langle \fu ,\varphi\rangle 
\qquad \forall \varphi \in \dD(\R^{d+1}_+).
\end{equation}
\emph{Step 2:} For $\alpha \in [0,1)$, we claim 
\begin{equation}
\|x_{d+1}^{-\alpha/2} \cA_\varepsilon(x_{d+1}^{\alpha/2} \varphi)\|_{L^2(\R^{d+1}_+)} \leq 
\|\varphi\|_{L^2(\R^{d+1}_+)} 
\qquad 
\forall \varphi \in \dD(\R^{d+1}_+). 
\end{equation}
To see this, we start by noting 
\begin{equation}
\label{eq:supp}
\sup_{\substack{(x,z)\\  x > \varepsilon,\, -\varepsilon <z <\varepsilon,\, x - 2 \varepsilon -z > 0}}
x^{-\alpha/2} (x-2\varepsilon - z)^{\alpha/2} \leq 1. 
\end{equation}
We observe $x_{d+1}^{\alpha/2} \varphi \in \dD(\R^{d+1}_+)$
and $\operatorname*{supp} \cA_\varepsilon( x_{d+1}^{\alpha/2} \varphi) 
\subset \R^d \times (\varepsilon,\infty)$ and write 
\begin{align*}
x_{d+1}^{-\alpha/2} \cA_\varepsilon(x_{d+1}^{\alpha/2} \varphi)(x) = 
x_{d+1}^{-\alpha/2} \int_{z \in B_\varepsilon(0)} \rho_\varepsilon(z) 
(x_{d+1} - 2\varepsilon - z_{d+1})^{\alpha/2} \varphi(x-2\varepsilon e_{d+1}-z)\,dz. 
\end{align*}
{}From \eqref{eq:supp} and $\rho_\varepsilon \ge 0$ we get 
\begin{align*}
\|x_{d+1}^{-\alpha/2} \cA_\varepsilon( x_{d+1}^{\alpha/2} \varphi)\|_{L^2(\R^{d+1}_+)} &\leq \|\cA_\varepsilon \varphi\|_{L^2(\R^{d+1}_+)} 
\leq \|\varphi\|_{L^2(\R^{d+1}_+)}.  
\end{align*}

\emph{Step 3:} For $\alpha \in [0,1)$ we have for every $\varepsilon >0$
\begin{equation}
\label{eq:grad-uepsilon}
\|\nabla \fu_\varepsilon\|_{L^2_\alpha(\R^{d+1}_+)} \leq C \|\nabla \fu\|_{L^2_\alpha(\R^{d+1}_+)}. 
\end{equation}
To see \eqref{eq:grad-uepsilon}, fix a bounded open  
$\omega \subset \R^{d+1}_+$. 
We compute for $\varphi \in \dD(\omega)$ and $\varepsilon > 0$, noting that 
$x_{d+1}^{\alpha/2} \varphi \in \dD(\omega)$,  
\begin{align*}
\left| \langle x_{d+1}^{\alpha/2}\nabla \fu_\varepsilon , \varphi\rangle
\right| & = 
\left| \langle \nabla \fu_\varepsilon , x_{d+1}^{\alpha/2}\varphi\rangle\right|
= 
\left| - \langle \fu_\varepsilon , \nabla( x_{d+1}^{\alpha/2}\varphi)\rangle\right| \\
&=
\left| - \langle \fu, \cA_\varepsilon \nabla( x_{d+1}^{\alpha/2}\varphi)\rangle\right| = 
\left| - \langle \fu, \nabla( \cA_\varepsilon( x_{d+1}^{\alpha/2}\varphi))\rangle\right| = 
\left| \langle \nabla \fu, \cA_\varepsilon( x_{d+1}^{\alpha/2}\varphi)\rangle\right| \\
&
\leq \|\nabla \fu \|_{L^2_\alpha(\R^{d+1}_+)} 
\|x_{d+1}^{-\alpha/2} \cA_\varepsilon (x_{d+1}^{\alpha/2} \varphi)\|_{L^2(\R^{d+1}_+)} 
\stackrel{\text{Step 2}}{\leq} 
\|\nabla \fu\|_{L^2_\alpha(\R^{d+1}_+)} \|\varphi\|_{L^2(\omega)}. 
\end{align*}
Combining this with the observation 
\begin{equation}
\label{eq:grad-uepsilon-10}
\|x_{d+1}^{\alpha/2} \nabla \fu_\varepsilon\|_{L^2(\omega)} = 
\sup_{\varphi \in \dD(\omega)} 
\frac{\langle x_{d+1}^{\alpha/2} \nabla \fu_\varepsilon,\varphi\rangle}{\|\varphi\|_{L^2(\omega)}} 
\end{equation}
gives us  $\|\nabla \fu_\varepsilon\|_{L^2_\alpha(\omega)} 
\leq \|\nabla \fu\|_{L^2_\alpha(\R^{d+1}_+)}$. The 
claim \eqref{eq:grad-uepsilon} now follows since $\omega$ is arbitrary. 

\emph{Step 4:} For $\alpha \in (-1,0]$ we have for 
every bounded open $\omega \subset \R^{d+1}_+$
the existence of $C_\omega > 0$ such that for $\varepsilon \in (0,1]$ 
\begin{align*}
\|\nabla \fu_\varepsilon\|_{L^2(\omega)} 
\leq C_\omega \|\nabla \fu\|_{L^2_\alpha(\R^{d+1}_+)}. 
\end{align*}
The proof follows by inspecting the procedure of Step~2 and essentially
using Step~2 with $\alpha = 0$ there. 

\emph{Step 5:} Steps~3 and 4 show that $\fu \in H^1_{\rm loc}(\R^{d+1}_+)$: 
Fix a bounded, open and connected $\omega \subset \R^{d+1}_+$. Fix
a $\varphi \in \dD(\omega)$ with  $(\varphi,1)_{L^2(\omega)} \ne 0$. Exploiting 
the norm equivalence 
$$
\|v\|_{H^1(\omega)} \sim \|\nabla v\|_{L^2(\omega)} 
+ |(\varphi,v)_{L^2(\omega)}| \qquad \forall v \in H^1(\omega)
$$
we infer from Steps~3, 4, and the observation 
$\lim_{\varepsilon\rightarrow 0} (\fu_\varepsilon,\varphi)_{L^2(\omega)}
 = \langle \fu,\varphi\rangle$
that $(\fu_\varepsilon)_{\varepsilon \in (0,1]}$ 
is uniformly bounded in $H^1(\omega)$. Thus, a subsequence converges
weakly in $H^1(\omega)$ and strongly in $L^2(\omega)$ 
to a limit, 
which is the representation of the distribution $\fu$ on $\omega$. 

\emph{Step 6:} Claim: For any bounded open $\omega \subset \R^{d+1}_+$
we have $\fu \in L^2_\alpha(\omega)$. It suffices to show norm bounds 
for bounded open sets of the form $\omega = \omega^0 \times (0,1)$ 
with $\omega^0 \subset \R^d$. 
For that, consider again the regularized
functions $\fu_\varepsilon$ and assume, additionally (with the aid of 
a cut-off function), that 
$\fu_\varepsilon(x,x_{d+1}) = 0$ for $x_{d+1} \ge 1$ and $x \in \R^d$. 
Then for $x_{d +1} \in (0,1)$ we have  
\begin{align}
\label{eq:lemma-BL1-fct-10}
\fu_\varepsilon(x,x_{d+1}) = - \int_{x_{d+1}}^1 \partial_{d+1} \fu_\varepsilon(x,t)\,dt. 
\end{align}
For $\alpha \in (-1,0]$, we square, multiply by $x_{d+1}^\alpha$, and integrate
to get 
\begin{align*}
\|\fu_\varepsilon\|^2_{L^2_\alpha(\omega)} 
\leq C \|\partial_{d+1} \fu_\varepsilon \|^2_{L^2(\omega)}. 
\end{align*}
Since 
$\|\partial_{d+1} \fu_\varepsilon \|_{L^2(\omega)}$ can be controlled
uniformly in $\varepsilon \in (0,1]$ by Steps~4, 5 the proof is complete
for $\alpha \in (-1,0]$. 
For $\alpha \in [0,1)$, we square \eqref{eq:lemma-BL1-fct-10}, 
use a Cauchy-Schwarz inequality on the right-hand side and integrate to get 
\begin{align*}
\|\fu_\varepsilon\|^2_{L^2(\omega)} 
\leq C_\alpha  \|\partial_{d+1} \fu_\varepsilon\|^2_{L^2(\omega)}. 
\end{align*}
Again, Steps~3, 5 allows us to control the right-hand side uniformly
in $\varepsilon$. 
\end{proof}

\begin{proof}[Proof of Lemma~\ref{lem:bl:dense}]
  We choose an open cover $\left( U_j \right)_{j\in\N}$ of $\R^{d+1}_+$ by bounded sets and a partition of unity
  $\left( \psi_j \right)_{j\in\N}$ subordinate to this cover. For $\fu\in\BL^1_\alpha(\R^{d+1}_+)$ we have
  $\fu\psi_j\in H^1_\alpha(\R^{d+1}_+)$, and according to Lemma~\ref{lem:density}, $\fu\psi_j$ can be approximated to arbitrary accuracy
  by a function $\varphi_j\in C^\infty(\overline{\R^{d+1}_+})\cap H^1_\alpha(\R^{d+1}_+)$ in the norm  of $H^1_\alpha(\R^{d+1}_+)$
  and hence also in the norm of $\BL^1_\alpha(\R^{d+1}_+)$. By construction, only a finite number of $\varphi_j$ overlap, and hence
  $\sum_{j=0}^\infty \varphi_j \in C^\infty(\overline{\R^{d+1}_+})\cap \BL^1_\alpha(\R^{d+1}_+)$.
\end{proof}
\begin{proof}[Proof of Lemma~\ref{lem:bl:trace}]
  Using an appropriate cut-off function, this is a simple consequence 
  of Lemmas~\ref{lem:bl:dense}, \ref{lem:mult:trace}, and~\ref{lem:trace}.
\end{proof}
\begin{proof}[Proof of Corollary~\ref{cor:poincare}]
  Due to the density result of Lemma~\ref{lem:bl:dense} and the definition of the trace operator, it suffices to show
  \begin{align*}
    \norm{\fu}{L^2_\alpha(K)} \lesssim \norm{\trace\,\fu}{L^2(K')} + \norm{\nabla \fu}{L^2_\alpha(K)}
  \end{align*}
  for all $\fu\in C^\infty(\overline{\R^{d+1}_+})$.
  Using the abbreviation $v(x) = \fu(x_1,\dots,x_d,x)$, we note that due to H\"older's inequality,
  \begin{align*}
    \abs{v(x)}^2 &\lesssim \abs{v(0)}^2 + \abs{\int_0^x v'(t)\,dt}^2
    \lesssim \abs{v(0)}^2 + x^{1-\alpha} \norm{v'}{L^2_\alpha(0,b_{d+1})}^2.
  \end{align*}
  Multiplying the last equation by $x^\alpha$ and integrating over $K$ finishes the proof.
\end{proof}
\section{$\mathcal{H}$-matrix approximability}
For any  subset $D\subset{\R^{d+1}_+}$ define the space
\begin{align*}
  \hH_h(D) := \{ \fu \in H^1_\alpha(D) \mid
    & \exists \wilde\fu\in \BL^1_\alpha(\R^{d+1}_+), \trace\,\wilde\fu \in\sS^1_0(\tT_h) \subset H^s_0(\R^d;\Omega)
    \text{ such that } \fu|_D = \wilde\fu|_D,\\
    &a(\wilde\fu,\vV) = 0 \text{ for all } \vV\in \BL^1_{\alpha,0}(\R^{d+1}_+) \}
\end{align*}
and the space with additional orthogonality
\begin{align*}
  \hH_{h,0}(D) := \{ \fu \in H^1_\alpha(D) \mid
    & \exists \wilde\fu\in \BL^1_\alpha(\R^{d+1}_+), \wilde\fu\in\sS^1_0(\tT_h) \subset H^s_0(\R^d;\Omega)
    \text{ such that } \fu|_D = \wilde\fu|_D,\\
    &a(\wilde\fu,\vV) = 0 \text{ for all } \vV\in \BL^1_{\alpha,0}(\R^{d+1}_+)\\
    &\text{ and all } \vV \in \BL^1_\alpha(\R^{d+1}_+)\text{ with } \trace\,\vV \in \sS^1_0(\tT_h) \subset H^s_0(\R^d;\Omega) 
  \text{ and } \supp(\trace\,\vV)\subset \overline D\cap\R^d \}.
\end{align*}
Define the cubes with side length $R$ (henceforth called ``box'') by 
\begin{align}\label{eq:dif:hbox}
  B_R:= B_R^0\times (0,R)\subset\R^{d+1}.
\end{align}
We say that two boxes $B_{R_1}$
and $B_{R_2}$ are concentric if their projections on $\R^d$, i.e., 
the corresponding cubes
$B_{R_1}^0$ and $B_{R_2}^0$, share the same barycenter and are concentric.
For $h>0$, we define on $H^1_\alpha(B_R)$ the norm
\begin{align*}
  \enorm{\fu}_{h,R}^2 := \left( \frac{h}{R} \right)^2 \norm{\nabla \fu}{L^2_\alpha(B_R)}^2
  + \frac{1}{R^2} \norm{\fu}{L^2_\alpha(B_R)}^2.
\end{align*}
We have the following Caccioppoli-type inequality.
\begin{lemma}\label{lem:caccioppoli}
  Let $\Omega\subset\R^d$ be a Lipschitz domain.
  Let $R\in (0,8\,\diam(\Omega))$, $\delta \in (0,1)$, and $h>0$ be such that $16h \leq \delta R$.
  Let $B_R$ and $B_{(1+\delta)R}$ be two concentric boxes. Then, there exists a constant
  $C>0$ depending only on $\Omega$, $d$, and the $\gamma$-shape regularity of $\tT_h$,
  such that for all $\fu \in \hH_{h,0}(B_{(1+\delta)R})$ it holds
  \begin{align}\label{eq:caccioppoli}
    \norm{\nabla \fu}{L^2_\alpha(B_R)} \leq C \frac{1+\delta}{\delta} \enorm{\fu}_{h,(1+\delta)R}.
  \end{align}
\end{lemma}
\begin{proof}
In the proof, various boxes will appear. They will always be assumed to be concentric to $B_R$.
  Choose a function $\eta \in W^{1,\infty}(\R^{d+1})$ 
with $(\trace\eta)|_\Omega\in \sS^1_0(\tT_h)$,
  $\eta \equiv 1$ on $B_R$, $\supp(\eta) \subset B_{(1+\delta/4)R}$,
  $0\leq\eta\leq 1$, and $\norm{\nabla \eta}{\infty} \lesssim \left( \delta R \right)^{-1}$. We calculate
  \begin{align*}
    \norm{\nabla \fu}{L^2_\alpha(B_R)}^2
    \leq \norm{\nabla (\eta \fu)}{L^2_\alpha(B_{(1+\delta)R})}^2
    &= \int_{B_{(1+\delta)R}}x_{d+1}^\alpha \nabla (\eta\fu)\cdot\nabla(\eta\fu)\,dx\\
    &= \int_{B_{(1+\delta)R}} x_{d+1}^\alpha \fu^2 \left( \nabla\eta \right)^2\,dx
    + \int_{B_{(1+\delta)R}} x_{d+1}^\alpha \nabla\fu \cdot\nabla(\eta^2\fu)\,dx.
  \end{align*}
  We first deal with the last integral on the right-hand side. Due to the support properties
  of $\eta$ and the orthogonality properties of space $\hH_{h,0}(B_{(1+\delta)R})$, we see
  \begin{align*}
    \int_{B_{(1+\delta)R}} x_{d+1}^\alpha \nabla\fu \cdot\nabla(\eta^2\fu)\,dx
    &= \int_{(1+\delta)R} x_{d+1}^\alpha \nabla\fu
    \cdot\nabla \left( \wilde\eta\eE(\trace(\eta^2\fu) - I_h\trace(\eta^2\fu))\right)\,dx\\
    &\leq \norm{\nabla \fu}{L^2_\alpha(B_{(1+\delta)R})} \cdot
    \norm{\nabla \left( \wilde\eta\eE(\trace(\eta^2\fu)
    - I_h\trace(\eta^2\fu))\right)}{L^2_\alpha(B_{(1+\delta)R})}
  \end{align*}
  where $\wilde\eta$ is a cut-off function
  with support contained $B_{(1+3\delta/4)R}^0 \times (0,3\delta R/4)$
  and $\wilde\eta\equiv 1$ on $B_{(1+\delta/2)R}^0\times (0,\delta R/2)$,
  such that $\norm{\nabla \wilde\eta}{\infty} \lesssim (\delta R)^{-1}$.
  Furthermore, $I_h:C(\overline{\Omega})\cap H^1_0(\Omega) \rightarrow S^{1}_0(\tT_h) $ is the usual nodal interpolation operator (extended by zero outside 
  $\Omega$).
  Then, using Lemma~\ref{lem:ext}, we obtain
  \begin{align*}
    & \norm{\nabla \left( \wilde\eta\eE(\trace(\eta^2\fu)
    - I_h\trace(\eta^2\fu))\right)}{L^2_\alpha(B_{(1+\delta)R})}
     \\ 
    & \quad \leq \norm{\nabla \left( \eE(\trace(\eta^2\fu)
    - I_h\trace(\eta^2\fu))\right)}{L^2_\alpha(B_{(1+\delta)R})}
     + \left( \delta R \right)^{-1} 
    \norm{\eE(\trace(\eta^2\fu)
    - I_h\trace(\eta^2\fu))}{L^2_\alpha(B_{(1+3\delta/4)R})}\\
    &\quad \lesssim \norm{\trace(\eta^2\fu)- I_h\trace(\eta^2\fu)}{H^s(\R^d)}
    + \left( \delta R \right)^{-s} \norm{\trace(\eta^2\fu)- I_h\trace(\eta^2\fu)}{L^2(\R^d)}.
  \end{align*}
  For $r\in[0,1]$ it holds
  \begin{align}\label{lem:cacc:eq4}
    \norm{\trace(\eta^2\fu)- I_h\trace(\eta^2\fu)}{H^r(\R^d)}^2
    \lesssim h^{4-2r} \sum_{K\in\tT_h} \snorm{\trace(\eta^2\fu)}{H^2(K)}^2,
  \end{align}
  and a short calculation, cf.~\cite{FMP_NumerMath_15}, and an inverse estimate show that
  \begin{align*}
    \snorm{\trace(\eta^2\fu)}{H^2(K)}^2
    &\lesssim \frac{1}{(\delta R)^2} \norm{\nabla\trace(\eta \fu)}{L^2(K)}^2
    + \frac{1}{(\delta R)^4} \norm{\trace \fu}{L^2(K)}^2\\
    &\lesssim \frac{h^{2s-2}}{(\delta R)^2} \norm{\trace (\eta \fu)}{H^s(K)}^2
    + \frac{1}{(\delta R)^4} \norm{\trace \fu}{L^2(K)}^2.
  \end{align*}
  By the support properties of $\eta$, the sum in~\eqref{lem:cacc:eq4} extends over
  elements $K\cap B_{(1+\delta/4)R}\neq\emptyset$. As $h\leq (\delta R)/16$, it holds
  $\bigcup_{K\cap B_{(1+\delta/4)R}^0\neq\emptyset}K \subset B_{(1+\delta/2)R}^0$.
  Then, using $h/(\delta R)\leq 1$, we conclude that
  \begin{align*}
    \norm{\nabla \left( \wilde\eta\eE(\trace(\eta^2\fu)
    - I_h\trace(\eta^2\fu))\right)}{L^2_\alpha(B_{(1+\delta)R})}
    \lesssim \frac{h}{\left( \delta R \right)}\norm{\trace (\eta \fu)}{H^s(\R^d)}
    + \frac{h^{2-s}}{(\delta R)^2}\norm{\trace \fu}{L^2(B_{(1+\delta/2)R}^0)}.
  \end{align*}
  Choosing a cut-off function $\eta_2$ with $\eta_2\equiv 1$ on $B_{(1+\delta/2)R}$ and
  support contained in $B_{(1+3\delta/4)R}$ and employing
  the multiplicative trace inquality from Lemma~\ref{lem:mult:trace} we see
  \begin{align*}
    \norm{\trace \fu}{L^2(B_{(1+\delta/2)R}^0)} &\leq \norm{\trace(\eta_2\fu)}{L^2(B_{(1+3\delta/4)R}^0)}
    \leq C_{\rm tr} \norm{\eta_2\fu}{L^2_\alpha(\R^{d+1})}^{s}
    \norm{\nabla (\eta_2\fu)}{L^2_\alpha(\R^{d+1})}^{1-s}\\
    &\lesssim \frac{1}{(\delta R)^{1-s}}\norm{\fu}{L^2_\alpha(B_{(1+\delta)R})}
    + \norm{\fu}{L^2_\alpha(B_{(1+\delta)R})}^s\norm{\nabla\fu}{L^2_\alpha(B_{(1+\delta)R})}^{1-s}. 
  \end{align*}
  Together with the boundedness of the trace operator asserted in 
  Lemma~\ref{lem:trace}, i.e.,  
  $\norm{\trace( \eta\fu)}{H^s(\R^d)} \lesssim \norm{\fu}{L^2_\alpha(B_{(1+\delta)R})} +
  \norm{\nabla(\eta\fu)}{L^2_\alpha(B_{(1+\delta)R})}$,
  this implies
  \begin{align*}
    \int_{B_{(1+\delta)R}} x_{d+1}^\alpha \nabla\fu \cdot\nabla(\eta^2\fu)\,dx
    &\lesssim \norm{\nabla \fu}{L^2_\alpha(B_{(1+\delta)R})} \cdot
    \Bigl( 
    \frac{h}{\delta R}\norm{\fu}{L^2_\alpha(B_{(1+\delta)R})} +
    \frac{h}{\delta R}\norm{\nabla(\eta\fu)}{L^2_\alpha(B_{(1+\delta)R})}\\
    &+ \frac{h^{2-s}}{(\delta R)^{2+(1-s)}}\norm{\fu}{L^2_\alpha(B_{(1+\delta)R})}
    + \frac{h^{2-s}}{(\delta R)^2}\norm{\fu}{L^2_\alpha(B_{(1+\delta)R})}^s\norm{\nabla\fu}{L^2_\alpha(B_{(1+\delta)R})}^{1-s}
    \Bigr). 
  \end{align*}
  The four products on the right-hand side are estimated with Young's inequality: the first
  three ones using the form $ab \leq \varepsilon a^2 + \frac{1}{4\varepsilon}b^2$, and
  the last one with exponents $\frac{2}{2-s}$ and $\frac{2}{s}$. We conclude
  that there are positive constants $C_1$ and $C_2$ such that
  \begin{align*}
   \norm{\nabla (\eta \fu)}{L^2_\alpha(B_{(1+\delta)R})}^2
   \leq C_1\frac{1}{(\delta R)^2}\norm{\fu}{L^2_\alpha(B_{(1+\delta)R})}^2
   + C_2\frac{h^2}{(\delta R)^2} \norm{\nabla\fu}{L^2_\alpha(B_{(1+\delta)R})}^2
   + \frac{1}{2}\norm{\nabla(\eta\fu)}{L^2_\alpha(B_{(1+\delta)R})}^2.
  \end{align*}
  Subtracting the last term from the left-hand side finishes the proof.
\end{proof}
Denote by $\Pi_{h,R}:\left( H^1_\alpha(B_R),\enorm{\cdot}_{h,R} \right)
\rightarrow\left( \hH_{h,0}(B_R),\enorm{\cdot}_{h,R} \right)$
the orthogonal projection. For $\kK_H$ a shape-regular triangulation
of $\R^{d+1}_+$ of mesh width $H$, denote by
$\Pi_H:H^1_\alpha(\R^{d+1}_+)\rightarrow\sS^1(\kK_H)$ the
operator from~\cite{NochettoOS_16}.
\begin{lemma}\label{lem:aprox:1}
  Let $\delta \in (0,1)$ and $R\in(0,4\,\diam(\Omega))$ be such that $16h \leq \delta R$.
  Let $B_R$, $B_{(1+\delta)R}$, and $B_{(1+2\delta)R}$ 
  be three concentric boxes. Let $\fu\in \hH_{h,0}(B_{(1+2\delta)R})$ 
  and suppose that $16H\leq \delta R$.
  Let $\eta \in C^\infty_0(\R^{d+1})$ with $\supp(\eta) \subset B_{(1+\delta)R}$
  and $\eta= 1$ on $B_R$.
  Then it holds
  \begin{enumerate}[(i)]
    \item
\label{item:lem:approx:1-i}
$(\fu - \Pi_{h,R}\Pi_H (\eta \fu))|_{B_R}\in \hH_{h,0}(B_R)$;
    \item
\label{item:lem:approx:1-ii}
$\enorm{\fu-\Pi_{h,R}\Pi_H(\eta\fu)}_{h,R} \leq
      C_{\rm app}\frac{1+2\delta}{\delta}\left( \frac{h}{R}+\frac{H}{R} \right)\enorm{u}_{h,(1+2\delta)R}$;
    \item
\label{item:lem:approx:1-iii}
$\dim W \leq C_{\rm app} \left( \frac{(1+2\delta)R}{H} \right)^{d+1}$,
      where $W := \Pi_{h,R}\Pi_H\eta \hH_{h,0}(B_{(1+2\delta)R})$.
  \end{enumerate}
\end{lemma}
\begin{proof}
To see (\ref{item:lem:approx:1-i}), 
note that if $\fu\in \hH_{h,0}(B_{(1+2\delta)R})$, then
  $\fu\in \hH_{h,0}(B_{R})$, and $\Pi_{h,R}$ maps into $\hH_{h,0}(B_{R})$.
To see (\ref{item:lem:approx:1-ii}), 
first note that due to the support properties of $\eta$ and the fact
  that $\Pi_{h,R}$ is the orthogonal projection, it holds
  \begin{align*}
    \enorm{\fu - \Pi_{h,R}\Pi_H(\eta\fu)}_{h,R}^2
    = \enorm{\Pi_{h,R}(\eta\fu - \Pi_H(\eta\fu))}_{h,R}^2 \leq
    \enorm{\eta\fu - \Pi_H(\eta\fu)}_{h,R}^2.
  \end{align*}
  Furthermore, due to the approximation properties of $\Pi_H$ 
  given in~\cite[Thm.~5.4]{NochettoOS_16}, we obtain
  \begin{align*}
    \enorm{\eta\fu - \Pi_H(\eta\fu)}_{h,R}^2
    &= \frac{h^2}{R^2} \norm{\nabla (\eta\fu-\Pi_H(\eta\fu))}{L^2_\alpha(B_R)}^2 + \frac{1}{R^2} \norm{\eta\fu-\Pi_H(\eta\fu)}{L^2_\alpha(B_R)}^2\\
    &\lesssim \left( \frac{h^2}{R^2} + \frac{H^2}{R^2} \right) \norm{\nabla(\eta\fu)}{L^2_\alpha(B_{(1+\delta)R})}^2\\
    &\lesssim \left( \frac{h^2}{R^2} + \frac{H^2}{R^2} \right) \left( \frac{1}{\delta^2R^2} \norm{\fu}{L^2_\alpha(B_{(1+\delta)R})}^2
    + \norm{\nabla\fu}{L^2_\alpha(B_{(1+\delta)R})}^2 \right). 
  \end{align*}
  Applying Lemma~\ref{lem:caccioppoli} with $\wilde\delta = \delta/(1+\delta)$ and $\wilde R = (1+\delta)R$, i.e.,
  $(1+\wilde\delta)\wilde R=(1+2\delta)R$, shows
  \begin{align*}
    \norm{\nabla\fu}{L^2_\alpha(B_{(1+\delta)R})}^2 \lesssim \frac{(1+2\delta)^2}{\delta^2} \enorm{\fu}_{h,(1+2\delta)R}^2. 
  \end{align*}
  Together with the trivial estimate $\norm{\fu}{L^2_\alpha(B_{(1+\delta)R})} \leq (1+2\delta)R \enorm{\fu}_{h,(1+2\delta)R}$
  we conclude  (\ref{item:lem:approx:1-ii}). 
Statement (\ref{item:lem:approx:1-iii})
follows from the local definition of the operator $\Pi_H$.
\end{proof}
\begin{lemma}\label{lem:aprox:2}
  Let $q,\kappa\in (0,1)$, $R\in(0,2\,\diam(\Omega))$, and $k\in\N$.
  Assume
  \begin{align}\label{lem:aprox:2:condition}
    h \leq \frac{\kappa qR}{64 k\max\left\{ 1,C_{\rm app} \right\}},
  \end{align}
  where $C_{\rm app}$ is the constant from Lemma~\ref{lem:aprox:1}.
  Then, there exists a finite dimensional subspace $\wat W_k$ of $\hH_h(B_{(1+\kappa)R})$ with dimension
  $\dim \wat W_k \leq C_{\rm dim} \left( \frac{1+\kappa^{-1}}{q} \right)^{d+1} k^{d+2}$ such that
  for every $\fu\in\hH_{h,0}(B_{(1+\kappa)R})$ it holds
  \begin{align*}
    \min_{\vV\in \wat W_k} \enorm{\fu-\vV}_{h,R} \leq q^k \enorm{\fu}_{h,(1+\kappa)R}.
  \end{align*}
\end{lemma}
\begin{proof}
  Define $H:=\frac{\kappa qR}{64k\max\left\{ 1,C_{\rm app} \right\}}$,
  then $h\leq H$.
  Define $\delta_j:=\kappa \frac{k-j}{k}$ for $j=0,\dots, k$. This yields $\kappa=\delta_0>\delta_1>\dots>\delta_k=0$.
  Now we will repeatedly apply Lemma~\ref{lem:aprox:1} $k$ times, with $\wilde R_j = (1+\delta_j)R$ and
  $\wilde\delta_j = \frac{1}{2k(1+\delta_j)}$. This can be done, as $\wilde R_j\leq 4\,\diam(\Omega)$, $\wilde\delta_j<1/2$, and
  \begin{align*}
    16 H \leq \frac{R}{4 k\max\left\{ 1,C_{\rm app} \right\}} \leq \frac{R}{2 k (1+\delta_j)} = \wilde\delta_j R\leq \wilde\delta_j \wilde R_j.
  \end{align*}
  Note that $(1+2\wilde\delta_j)\wilde R_j = (1+\delta_{j-1})R$. The first application of Lemma~\ref{lem:aprox:1}
  yields a function $\fw_1$ in a subspace $\wat W_1$ of $\hH_h(B_{(1+\delta_1)R})$ with
  $\dim W_1 \leq C_{\rm app}\left( \frac{(1+\kappa)R}{H} \right)^{d+1}$ such that
  \begin{align*}
    \enorm{\fu-\fw_1}_{h,(1+\delta_1)R}
    \leq 2C_{\rm app} \frac{1+2\wilde\delta_1}{\wilde\delta_1}\frac{H}{\wilde R_1} \enorm{\fu}_{h,(1+2\wilde\delta_1)\wilde R_1}
    = 8C_{\rm app} \frac{kH}{R} \enorm{\fu}_{h,(1+\delta_0)R}
    \leq q \enorm{\fu}_{h,(1+\delta_0)R}.
  \end{align*}
  As $\fu-\fw_1\in \hH_h(B_{(1+\delta_1)R})$, a second application of Lemma~\ref{lem:aprox:1} yields a function
  $\fw_2$ in a subspace $W_2$ of $\hH_h(B_{(1+\delta_2)R})$ such that
  \begin{align*}
    \enorm{\fu-\fw_1-\fw_2}_{h,(1+\delta_2)R}
    \leq q \enorm{\fu-\fw_1}_{h,(1+\delta_1)R} \leq q^2 \enorm{\fu}_{h,(1+\delta_0)R}.
  \end{align*}
  Applying $k$ times Lemma~\ref{lem:aprox:1}, we obtain a function $\vV=\sum_{j=1}^k \fw_j$
  that is an element of the subspace $V_k :=\sum_{j=1}^k W_j$ of $\hH_h(B_R)$ and fulfills
  $\enorm{\fu-\vV}_{h,R} \leq q^k \enorm{\fu}_{h,(1+\kappa)R}$.
\end{proof}
\begin{proposition}\label{prop:main}
  Let $\eta>0$ be a fixed admissibility parameter and $q\in (0,1)$.
  Let $(\tau,\sigma)$ be a cluster pair with admissible bounding boxes $B_{R_\tau}^0$ and $B_{R_\sigma}^0$,
  that is, $\eta\,\dist \left( B_{R_\tau}^0, B_{R_\sigma}^0 \right)\geq \diam\left( B_{R_\tau}^0 \right)$.
  Then, for each $k\in\N$ there exists a space $V_k\subset \sS^1_0(\tT_h)$ with
  $\dim V_k\leq C_{\rm dim}(2+\eta)^{d+1}q^{-(d+1)}k^{d+2}$ such that if $f\in L^2(\Omega)$ with $\supp(f)\subset B_{R_\sigma}^0\cap\Omega$,
  the solution $u_h$ of~\eqref{eq:discrete} satisfies
  \begin{align}\label{prop:main:eq}
    \min_{v \in V_k} \norm{u_h - v}{L^2(B_{R_\tau}^0)} \leq C_{\rm box} h^{-1}q^k \norm{f}{L^2(B_{R_\sigma}^0)}.
  \end{align}
\end{proposition}
\begin{proof}
  Set $\kappa := (1+\eta)^{-1}$. We distinguish two cases.
  \begin{enumerate}
    \item \textbf{Condition~\eqref{lem:aprox:2:condition} is satisfied with $R=R_\tau$:}
      As $\dist(B_{R_\tau},B_{R_\sigma}) \geq \eta^{-1}\diam(B_{R_\tau})=\eta{-1}\sqrt{d}R_\tau$, we conclude
      \begin{align*}
	\dist(B_{(1+\kappa)R_\tau},B_{R_\sigma}) \geq \dist(B_{R_\tau},B_{R_\sigma}) - \kappa R_\tau \sqrt{d}
	= \sqrt{d}R_\tau \frac{1}{\eta(1+\eta)}>0,
      \end{align*}
      hence $\lL u_h \in \hH_{h,0}(B_{(1+\kappa)R})$. Lemma~\ref{lem:aprox:2} implies that there is a
      space $\wat W_k$ with 
      \begin{align*}
        \min_{\vV\in \wat W_k} \enorm{\lL u_h-\vV}_{h,R_\tau} \leq q^k \enorm{\lL u_h}_{h,(1+\kappa)R_\tau}.
      \end{align*}
      Now
      \begin{align*}
	\enorm{\lL u_h}_{h,(1+\kappa)R_\tau} \lesssim (1+\frac{1}{R_\tau})\norm{\lL u_h}{\BL^1_\alpha(\R^{d+1}_+)}
	\lesssim (1+\frac{1}{R_\tau}) \norm{\Pi f}{L_2(\Omega)},
      \end{align*}
      where the last estimate follows from~\eqref{eq:1} and Lemma~\ref{lem:dirichlet:var}.
      On the other hand, employing an appropriate cut-off function 
      and the multiplicative trace estimate of Lemma~\ref{lem:mult:trace} shows
      \begin{align*}
	\norm{u_h - \trace\,\vV}{L^2(B_{R_\tau}^0)} \lesssim \frac{1}{R_\tau} \norm{\lL u_h - \vV}{L^2_\alpha(B_{R_\tau})} +
	\norm{\nabla ( \lL u_h - \vV)}{L^2_\alpha(B_{R_\tau})} \lesssim \frac{R_\tau}{h} \enorm{\lL u_h -\vV}_{h,R_\tau}.
      \end{align*}
      Combining the last three chains of estimates, we get the desired result if we set $V_k := \trace\,\wat W_k$.
    \item \textbf{Condition~\eqref{lem:aprox:2:condition} is not satisfied with $R=R_\tau$:}
      We select $V_k := \left\{ v|_{B_{R_\tau}^0} \mid v \in \sS^1_0(\tT_h) \right\}$. The minimum in~\eqref{prop:main:eq}
      is then zero and
      \begin{align*}
	\dim V_k \lesssim \left( \frac{R_\tau}{h} \right)^d \leq \left( \frac{64 k \max\left\{ 1,C_{\rm app} \right\}}{\kappa q} \right)^d
	\lesssim k^d(1+\eta)^d q^{-d}.
      \end{align*}
  \end{enumerate}
\end{proof}
\begin{proof}[Proof of Theorem~\ref{thm:approx:blocks}]
  Suppose first that $C_{\rm dim}(2+\eta)^{d+1}q^{-(d+1)}k^{d+2} \geq \min\left\{ \abs{\tau},\abs{\sigma} \right\}$.
  In the case $\min\left\{ \abs{\tau},\abs{\sigma} \right\}=\abs{\tau}$ we set
  $\Xf_{\tau\sigma}=\If\in\R^{\abs{\tau}\times\abs{\tau}}$ and $\Yf_{\tau\sigma} = \Af^{-1}|_{\tau\times\sigma}^\top$.
  If $\min\left\{ \abs{\tau},\abs{\sigma} \right\}=\abs{\sigma}$, we set
  $\Xf_{\tau\sigma} = \Af^{-1}|_{\tau\times\sigma}$ and $\Yf_{\tau\sigma} = \If\in\R^{\abs{\sigma}\times\sigma}$.
  Now suppose that  $C_{\rm dim}(2+\eta)^{d+1}q^{-(d+1)}k^{d+2} < \min\left\{ \abs{\tau},\abs{\sigma} \right\}$.
  For a cluster $\tau\subset\Ii$ we define $\R^\tau := \left\{ \xx\in\R^N\mid \xx_j = 0\; \forall j\notin \tau  \right\}$.
  According to~\cite{ScottZ_90}, there exist linear functionals $\lambda_i$ such that $\lambda_i(\psi_j)=\delta_{ij}$,
  and
  \begin{align}\label{thm:approx:blocks:eq1}
    \norm{\lambda_i(w)\psi_i}{L_2(\Omega)} \lesssim \norm{w}{L_2(\supp(\psi_i))},
  \end{align}
  where the hidden constant depends only on the shape-regularity of $\tT_h$. Define
  \begin{align*}
  \Phi_\tau:
  \begin{cases}
    \R^\tau \rightarrow \sS^1_0(\tT_h)\\
    \xx \mapsto \sum_{j\in\tau}\xx_j\psi_j
  \end{cases}
  \quad\text{ and }\quad
  \Lambda_\tau:
  \begin{cases}
    L_2(\Omega)\rightarrow \R^\tau\\
    w\mapsto \ww
  \end{cases},
  \end{align*}
  where $\ww_j = \lambda_j(w)$ for $j\in\tau$ and $\ww_j=0$ else.
  Note that $h^{d/2}\norm{\xx}{2}\sim\norm{\Phi_\tau(\xx)}{L_2(\Omega)}$ for $\xx\in\R^\tau$,
  and that $\Phi_\tau\circ\Lambda_\tau$ is bounded in $L_2(\Omega)$.
  For $\Lambda_\Ii^\star$ the adjoint of $\Lambda_\Ii$, this implies $\norm{\Lambda_\Ii^\star}{\R^N\rightarrow L_2(\Omega)}\lesssim h^{-d/2}$.
  Let $V_k$ be the space of Proposition~\ref{prop:main}.
  We define the columns of $\Xf_{\tau\sigma}$ to be a orthogonal basis of the space $\left\{ \Lambda_\tau w \mid w\in V_k \right\}$
  and $\Yf_{\tau\sigma} := \Af^{-1}|_{\tau\times\sigma}^\top \Xf_{\tau\sigma}$. The ranks of $\Xf_{\tau\sigma}$ and $\Yf_{\tau\sigma}$
  are then bounded by $C_{\rm dim}(2+\eta)^{d+1}q^{-(d+1)}k^{d+2}$.
  Now, for $\bb\in\R^\sigma$, set $f := \Lambda_\Ii^\star\left( \bb \right)$. This yields $b_i = \vdual{f}{\psi_i}_{\Omega}$ and
  $\supp(f)\subset \overline{B_{R_\sigma}\cap\Omega}$. According to Proposition~\ref{prop:main}, there exists
  an element $v\in V_k$ such that
  $\norm{u_h - v}{L_2(B_{R_\tau}^0\cap\Omega)} \lesssim h^{-1}q^k \norm{f}{L_2(B_{R_\sigma}^0)}$.
  This implies
  \begin{align*}
    \norm{\Lambda_\tau u_h - \Lambda_\tau v}{2}
    &\lesssim h^{-d/2}\norm{\Phi_\tau\circ\Lambda_\tau(u_h-v)}{L_2(\Omega)}
    \lesssim h^{-d/2}\norm{u_h - v}{L_2(B_{R_\tau}^0\cap\Omega)}\\
    &\lesssim h^{-1-d/2}q^k \norm{f}{L_2(B_{R_\sigma}^0)} \lesssim h^{-1-d}q^k\norm{\bb}{2}.
  \end{align*}
  For $\zz := \Xf_{\tau\sigma}\Xf_{\tau\sigma}^\top\Lambda_\tau u_h$ it holds
  \begin{align*}
    \norm{\Lambda_\tau u_h - \zz}{2} \leq \norm{\Lambda_\tau u_h - \Lambda_\tau v}{2} \lesssim h^{-1-d}q^k\norm{\bb}{2}.
  \end{align*}
  As $\Lambda_\tau u_h = \Af^{-1}|_{\tau\times\sigma}\bb|_\sigma$, we obtain
  \begin{align*}
    \norm{(\Af^{-1}|_{\tau\times\sigma} - \Xf_{\tau\sigma}\Yf_{\tau\sigma}^\top)\bb|_\sigma}{2} \lesssim N^{\frac{1+d}{d}} q^k \norm{\bb}{2}.
  \end{align*}
  As $\bb\in\R^\sigma$ was arbitrary, the result follows.
\end{proof}
\section{Numerical experiments}
We provide numerical experiments in two space dimensions, i.e.,  $d=2$, that confirm 
our theoretical findings. The indices $\Ii$ of the standard basis of the space $S^1_0(\tT_h)$ 
based on a quasiuniform triangulation of $\Omega$, are organized in a cluster tree
$\mathbb{T}_{\Ii}$ that is obtained by 
a geometric clustering, i.e., bounding boxes are split 
in half perpendicular to their longest edge until the corresponding clusters are smaller than $n_{\rm leaf}=20$. The block cluster tree is based on that cluster tree using 
the admissibility parameter $\eta=2$. 
In order to calculate a blockwise rank-$r$ approximation $\Bf_{\hH}^r$ of $\Af^{-1}$, we compute
the densely populated system matrix $\Af$ using the MATLAB code presented in~\cite{AcostaBB_17}.
On admissible cluster pairs, we compute a rank-$r$ approximation of the corresponding matrix block by singular
value decomposition. We carried out experiments for $s\in\left\{ 0.25, 0.5, 0.75 \right\}$ on a square and an L-shaped domain.
On the square, we use a coarse mesh of 2674 elements, resulting in 358 admissible and 591 non-admissible blocks,
and a fine mesh of 17130 elements, resulting in 5234 admissible and 5486 non-admissible blocks. 
On the L-shaped domain,
we use a mesh of 6560 elements, resulting in 640 admissible and 1332 non-admissible blocks.
Note that for a fixed mesh and cluster tree, Theorem~\ref{th:Happrox} predicts $\norm{\Af^{-1}-\Bf_{\hH}^r}{2}\lesssim e^{-br^{1/4}}$.
However, in our experiments we observe that the error behaves like $\norm{\Af^{-1}-\Bf_{\hH}^r}{2} \sim e^{-10 r^{1/3}}$.
Hence, we will plot the error logarithmically over the third root of the block rank $r$, and include
the reference curve $e^{-10 r^{1/3}}$.
\begin{figure}[htb]
  \centering
  \psfrag{err}{$\norm{\Af^{-1}-\Bf_{\hH}^r}{2}$}
  \psfrag{ref}{$\exp(-10r^{1/3})$}
  \includegraphics[width=\textwidth]{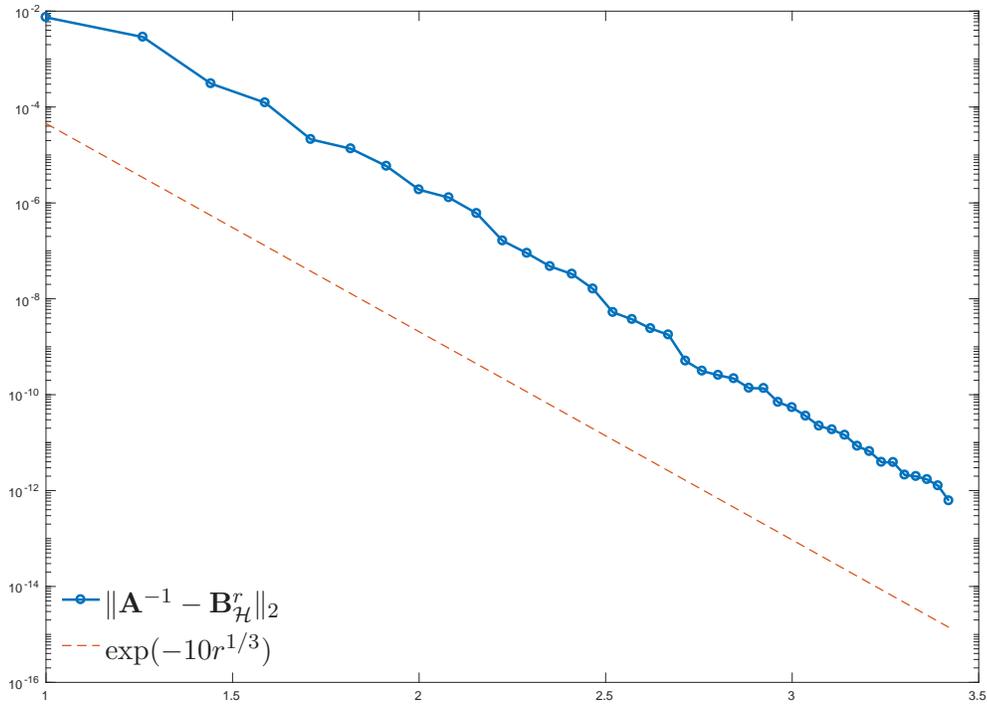}
  \caption{Square domain, $s=0.25$, 2674 elements.}
  \label{sc025}
\end{figure}
\begin{figure}[htb]
  \centering
  \psfrag{err}{$\norm{\Af^{-1}-\Bf_{\hH}^r}{2}$}
  \psfrag{ref}{$\exp(-10r^{1/3})$}
  \includegraphics[width=\textwidth]{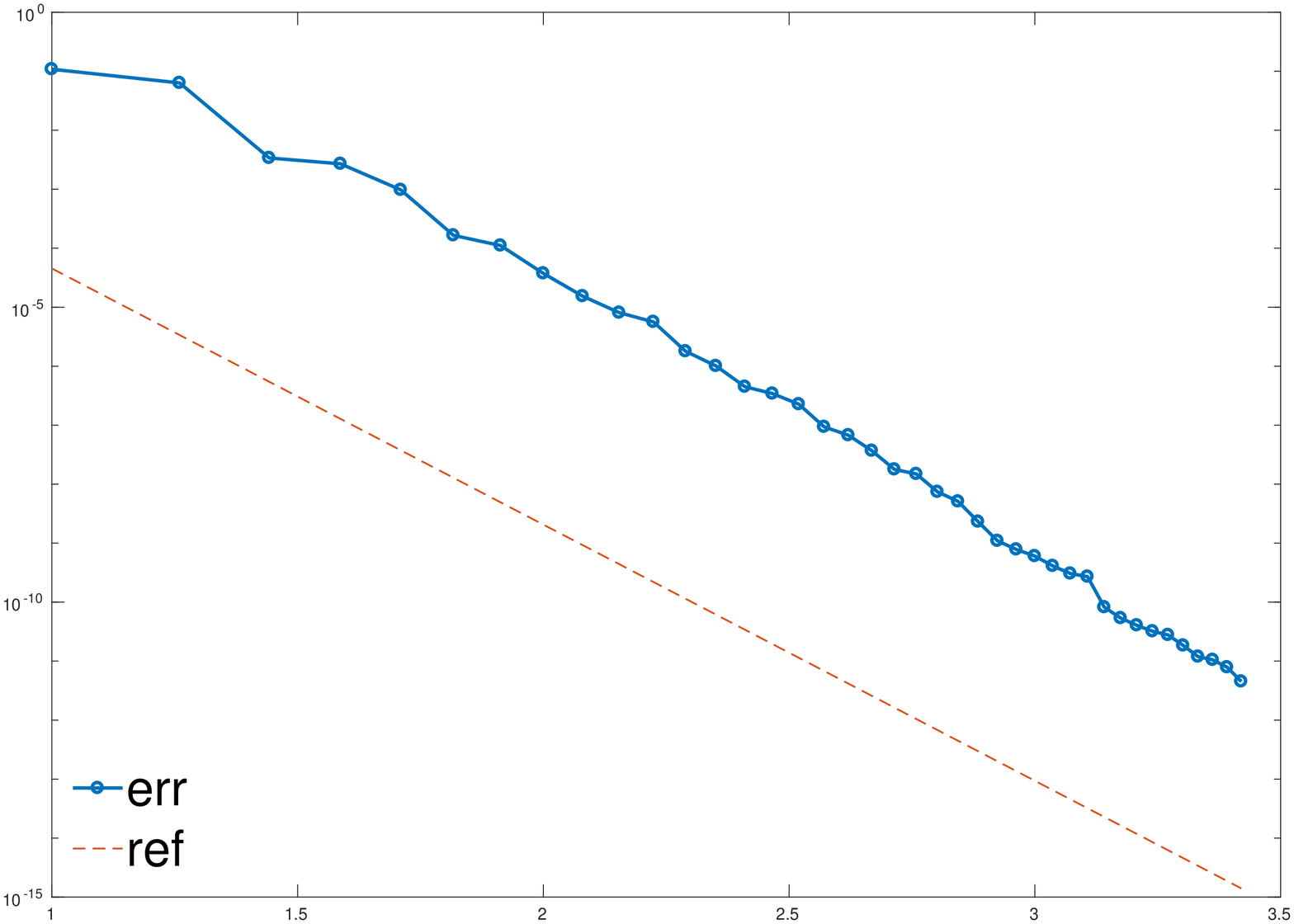}
  \caption{Square domain, $s=0.25$, 17130 elements.}
  \label{sf025}
\end{figure}
\begin{figure}[htb]
  \centering
  \psfrag{err}{$\norm{\Af^{-1}-\Bf_{\hH}^r}{2}$}
  \psfrag{ref}{$\exp(-10r^{1/3})$}
  \includegraphics[width=\textwidth]{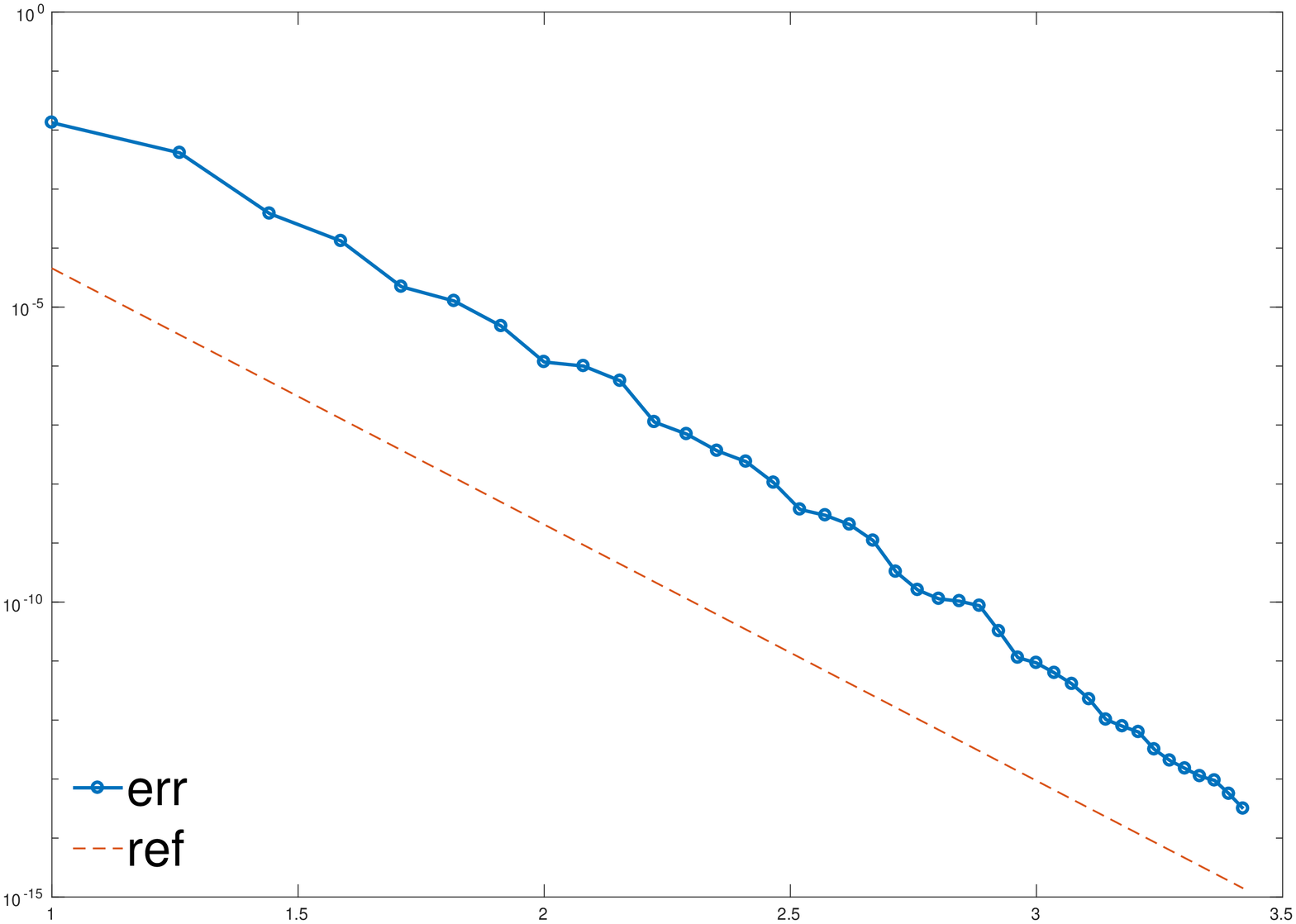}
  \caption{Square domain, $s=0.5$, 2674 elements.}
  \label{sc05}
\end{figure}
\begin{figure}[htb]
  \centering
  \psfrag{err}{$\norm{\Af^{-1}-\Bf_{\hH}^r}{2}$}
  \psfrag{ref}{$\exp(-10r^{1/3})$}
  \includegraphics[width=\textwidth]{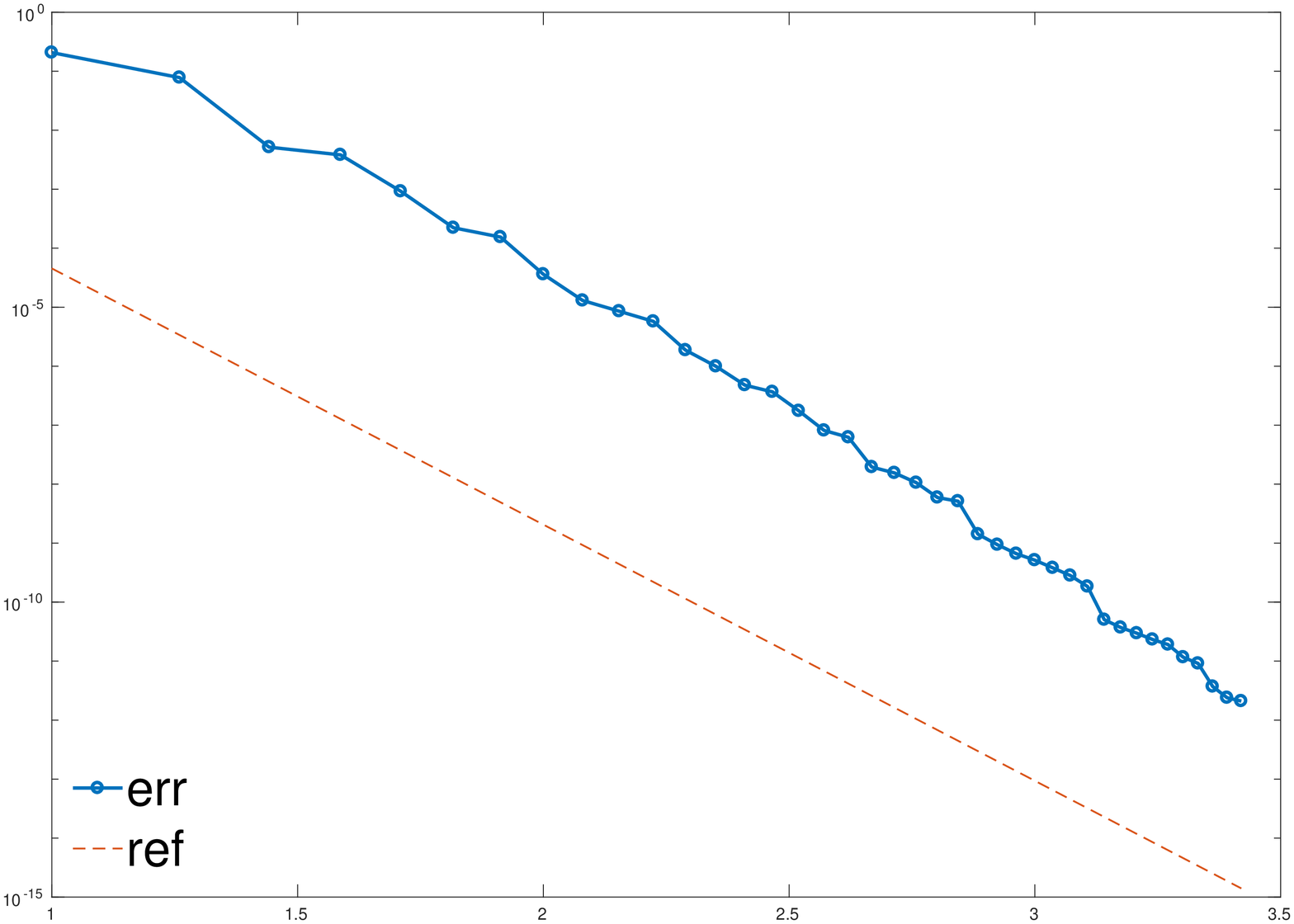}
  \caption{Square domain, $s=0.5$, 17130 elements.}
  \label{sf05}
\end{figure}
\begin{figure}[htb]
  \centering
  \psfrag{err}{$\norm{\Af^{-1}-\Bf_{\hH}^r}{2}$}
  \psfrag{ref}{$\exp(-10r^{1/3})$}
  \includegraphics[width=\textwidth]{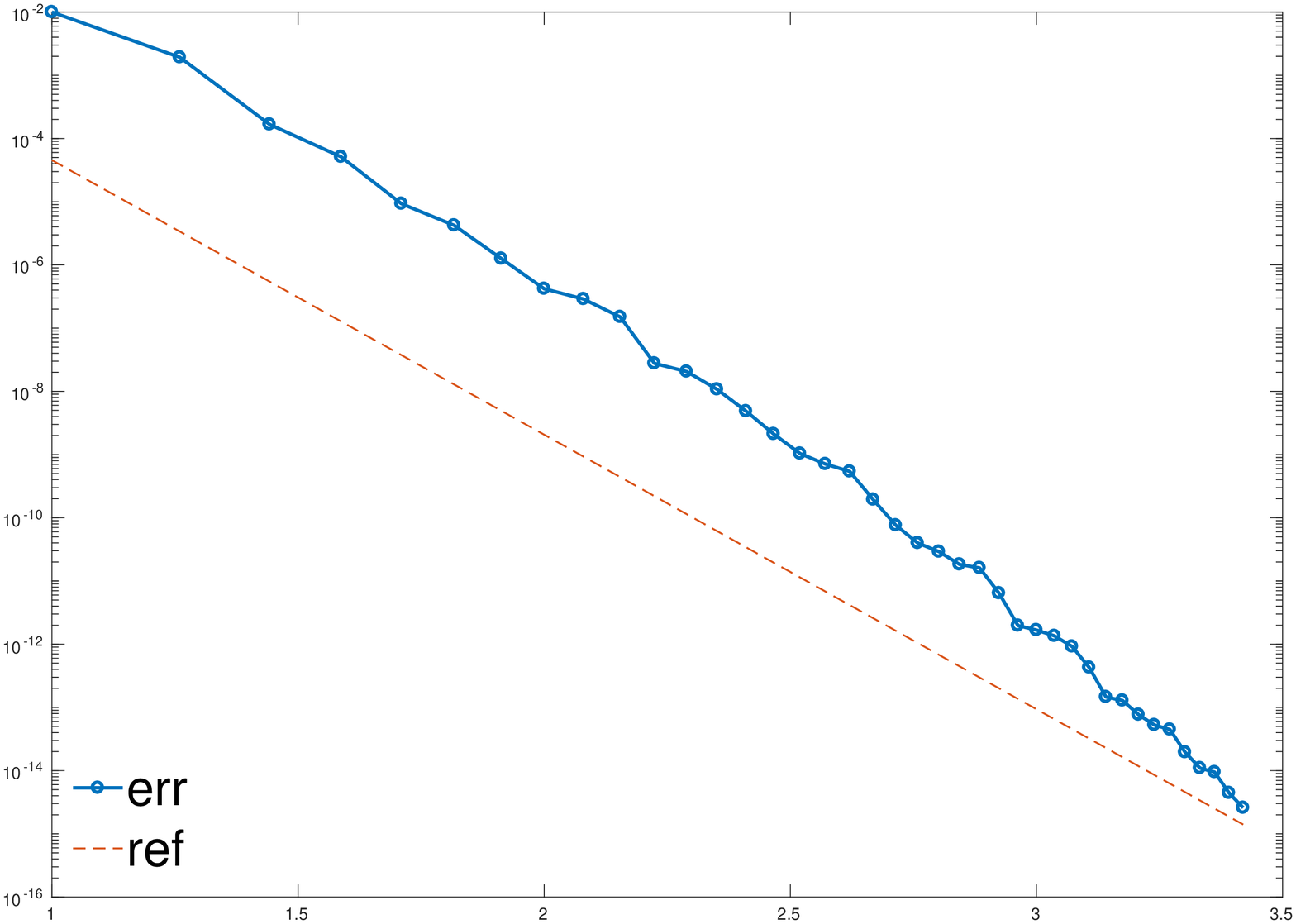}
  \caption{Square domain, $s=0.75$, 2674 elements.}
  \label{sc075}
\end{figure}
\begin{figure}[htb]
  \centering
  \psfrag{err}{$\norm{\Af^{-1}-\Bf_{\hH}^r}{2}$}
  \psfrag{ref}{$\exp(-10r^{1/3})$}
  \includegraphics[width=\textwidth]{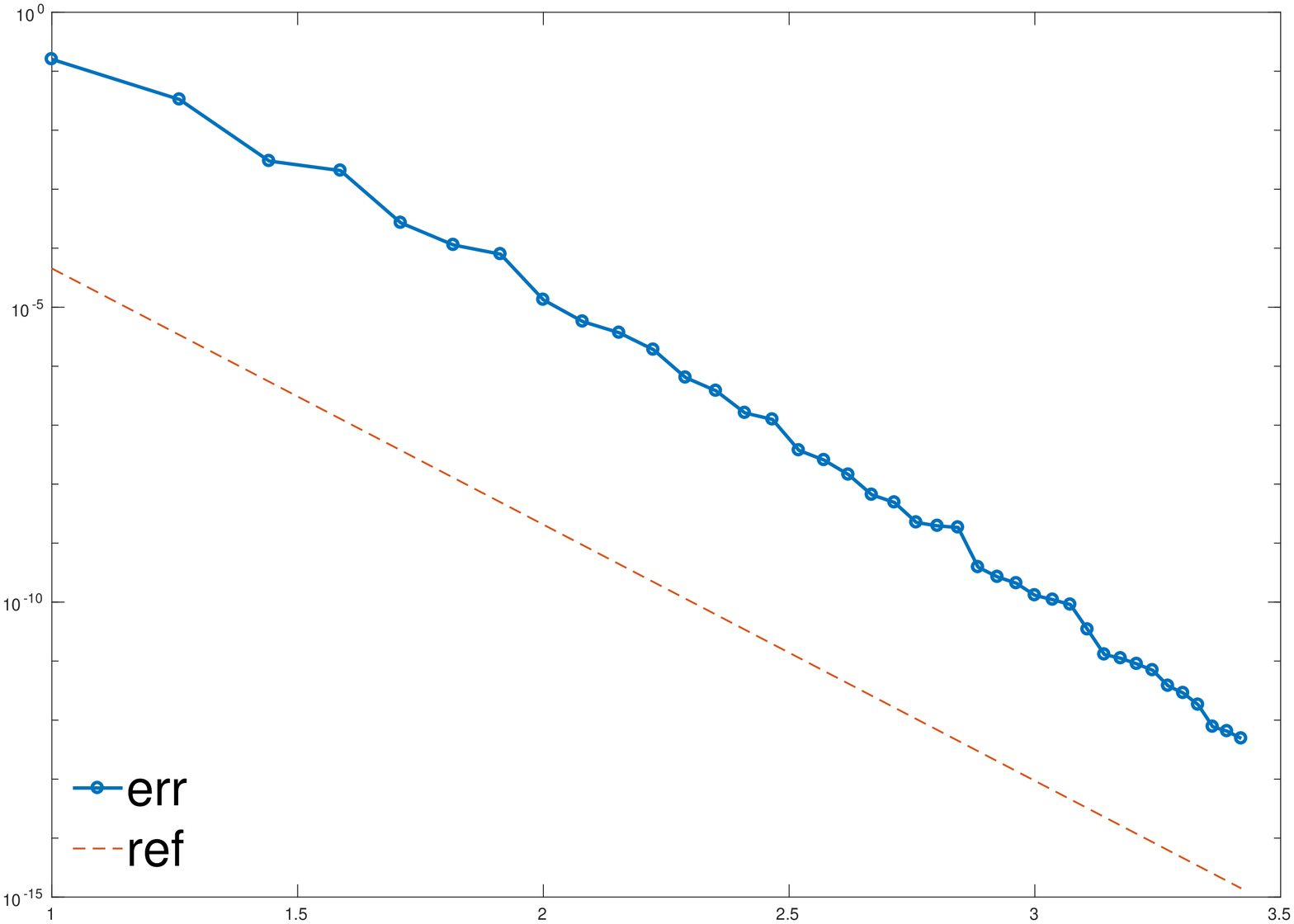}
  \caption{Square domain, $s=0.75$, 17130 elements.}
  \label{sf075}
\end{figure}
\begin{figure}[htb]
  \centering
  \psfrag{err}{$\norm{\Af^{-1}-\Bf_{\hH}^r}{2}$}
  \psfrag{ref}{$\exp(-10r^{1/3})$}
  \includegraphics[width=\textwidth]{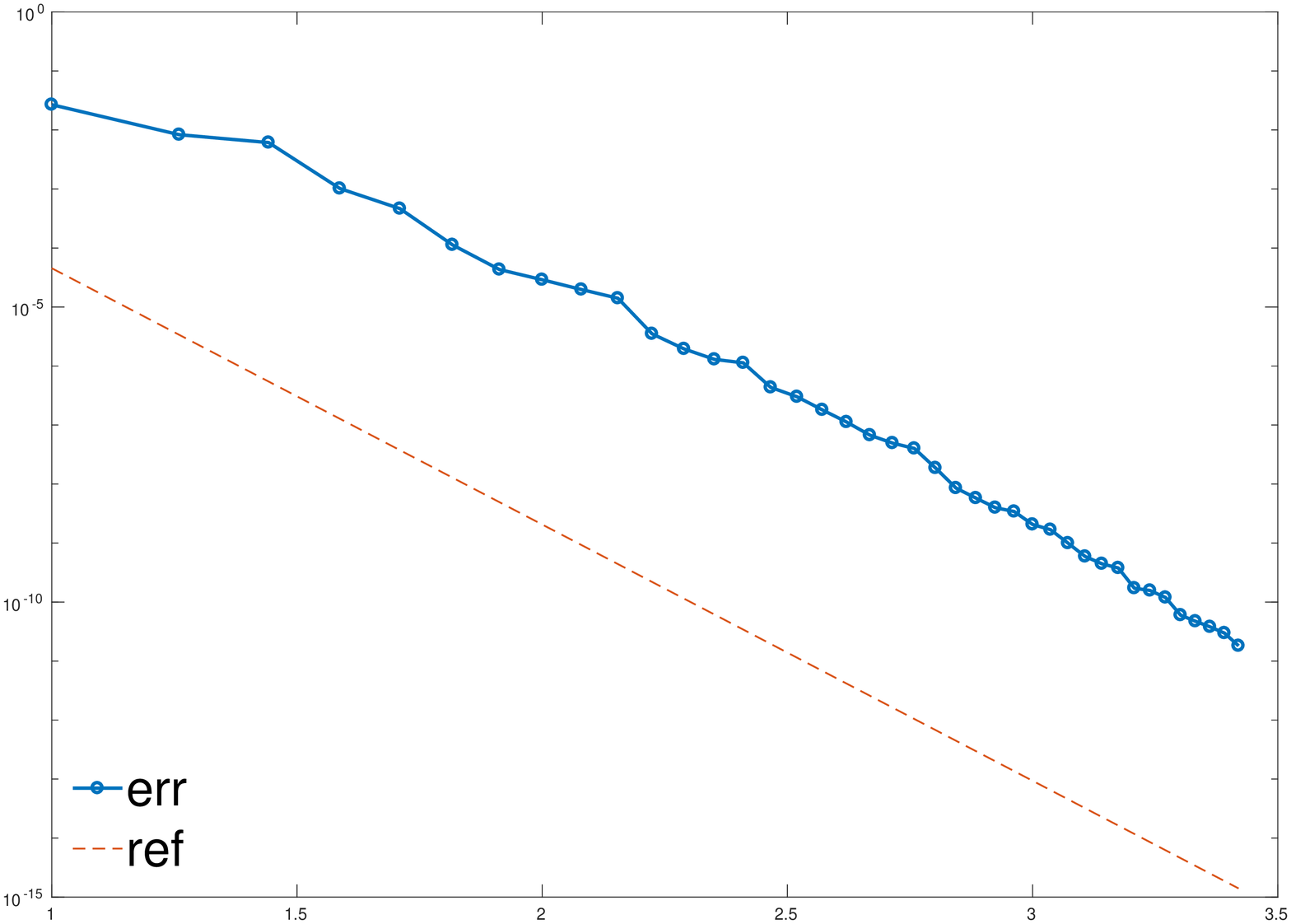}
  \caption{L-shaped domain, $s=0.25$, 6560 elements.}
  \label{l025}
\end{figure}
\begin{figure}[htb]
  \centering
  \psfrag{err}{$\norm{\Af^{-1}-\Bf_{\hH}^r}{2}$}
  \psfrag{ref}{$\exp(-10r^{1/3})$}
  \includegraphics[width=\textwidth]{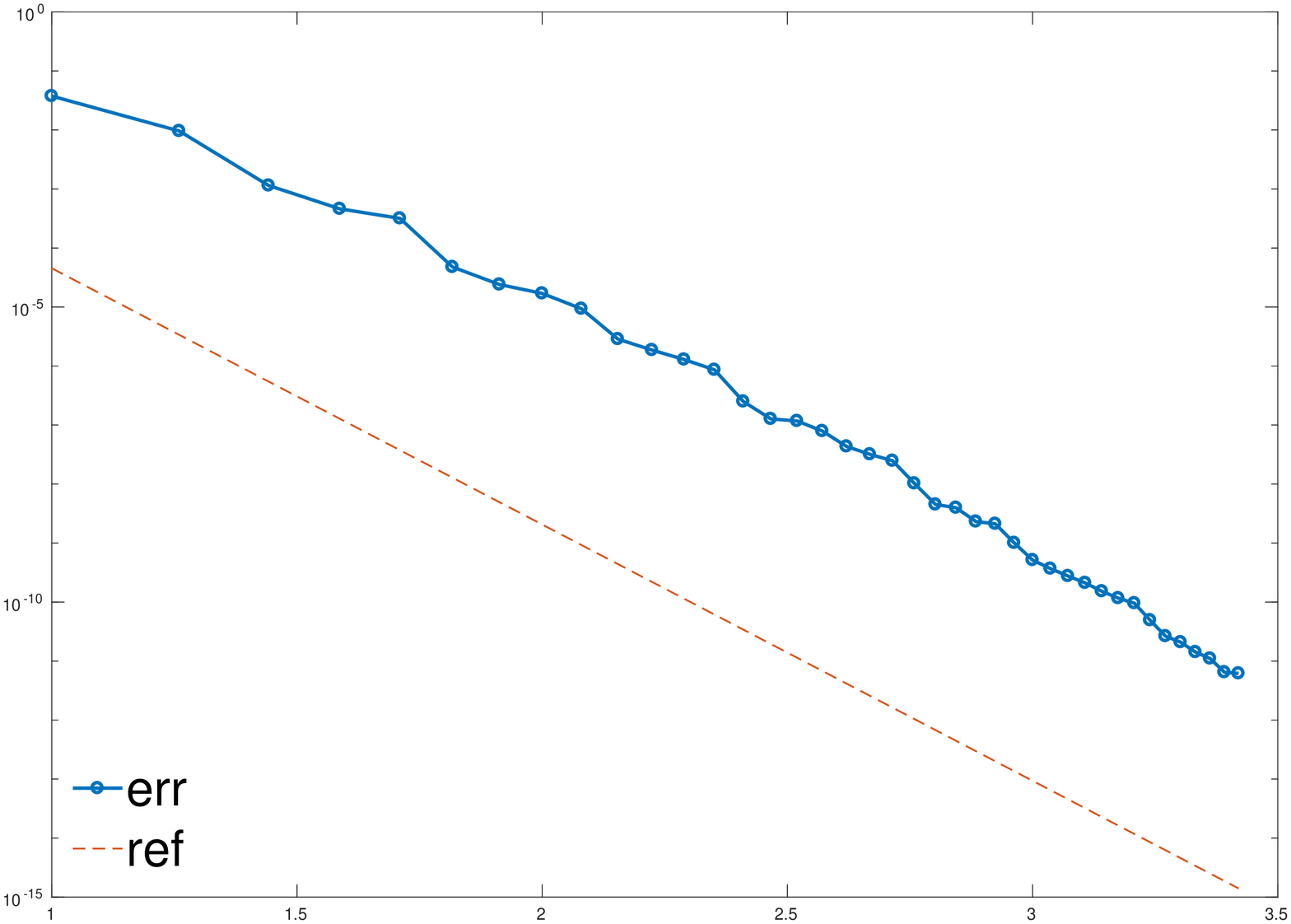}
  \caption{L-shaped domain, $s=0.5$, 6560 elements.}
  \label{l05}
\end{figure}
\begin{figure}[htb]
  \centering
  \psfrag{err}{$\norm{\Af^{-1}-\Bf_{\hH}^r}{2}$}
  \psfrag{ref}{$\exp(-10r^{1/3})$}
  \includegraphics[width=\textwidth]{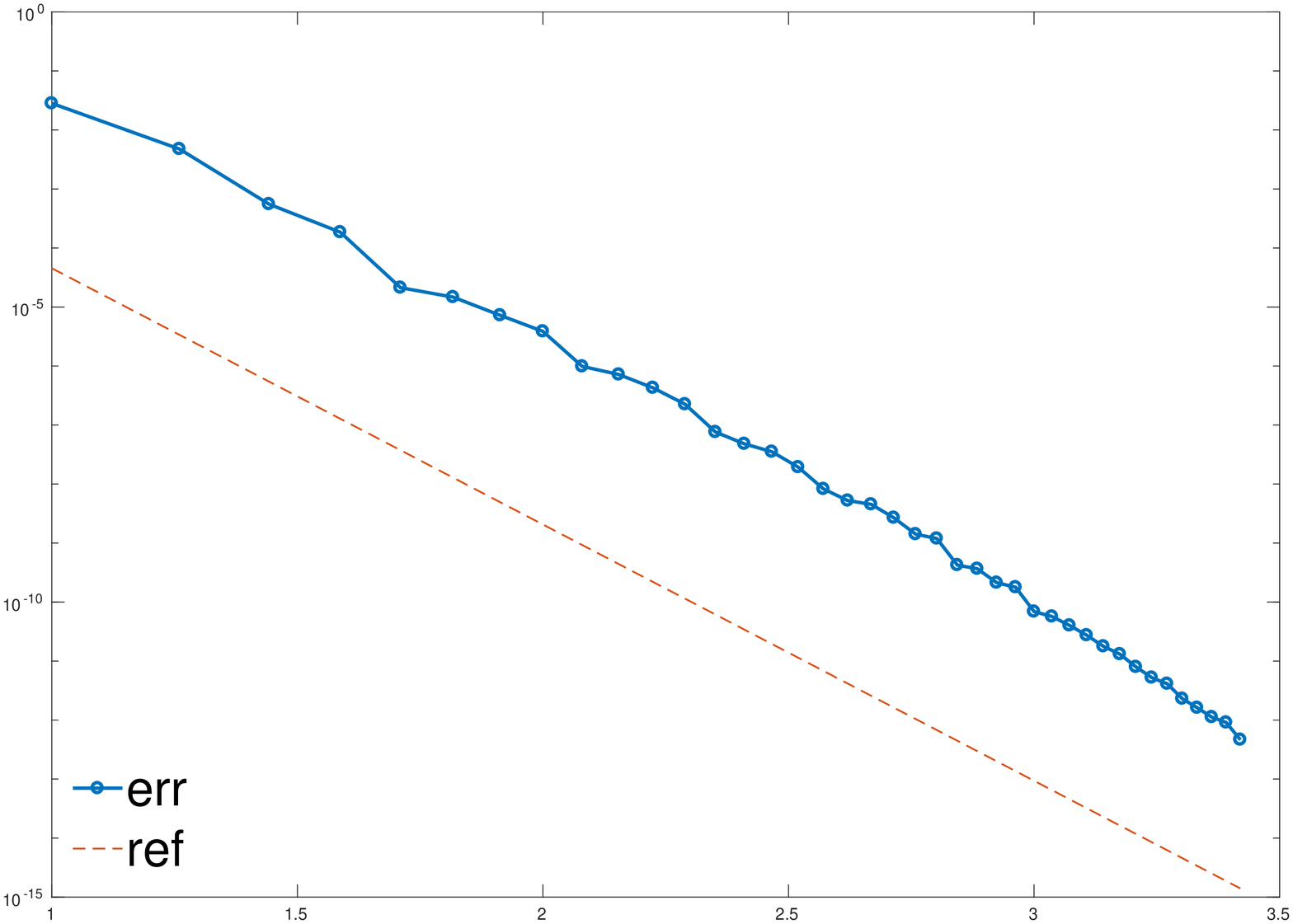}
  \caption{L-shaped domain, $s=0.75$, 6560 elements.}
  \label{l075}
\end{figure}

\section{Conclusions and extensions}
We have shown that the inverse $\Af^{-1}$ of the stiffness matrix $\Af$ of a 
Galerkin discretization of the fractional Laplacian can be approximated
at an exponential rate in the block rank by $\calH$-matrices, using the 
standard admissibility criterion \eqref{eq:admissiblity}. 
The following extensions are possible:
\begin{itemize}
\item We restricted our analysis to 
the discretization by piecewise linears. However, the analysis generalizes
to approximation by piecewise polynomials of fixed degree $p$. 
\item 
We focussed on the approximability of $\Af^{-1}$
in the $\calH$-format. Computationally attractive are also factorizations
such as $\calH$-LU or $\calH$-Cholesky factorizations. The ability to find
an approximate $\Af \approx L_{\calH} U_{\calH}$ has been shown 
for (classical) FEM discretizations in \cite{Bebendorf07,FMP_NumerMath_15}
and for non-local BEM matrices in \cite{faustmann-melenk-praetorius15b,faustmann-melenk-praetorius15c} with techniques that generalize to the present 
case of the fractional Laplacian. 
\item Related to $\calH$-matrices is the format
of $\calH^2$-matrices discussed in  
\cite{HackbuschKhoromskijSauter,BoermBuch,hackbusch15,Boerm}. 
Using the techniques employed 
in \cite{Boerm,FMP_NumerMath_15,faustmann-melenk-praetorius15b,faustmann-melenk-praetorius15c}, one may also show that $\Af^{-1}$ can be approximated 
by $\calH^2$-matrices   at an exponential rate in the block rank. 
\end{itemize}
\bibliographystyle{abbrv}
\bibliography{literature}
\end{document}